\documentclass[12pt]{amsart}
\usepackage{amsmath}
\usepackage[margin=2.5cm]{geometry}
\usepackage{amssymb}
\usepackage{mathtools}
\usepackage{faktor}
\usepackage{dsfont}
\usepackage{array}
\usepackage{graphicx,caption,color}
\usepackage[dvipsnames]{xcolor}
\usepackage{tikz}
\usetikzlibrary{cd}
\usepackage{babel}
\usepackage{hyperref}
\usepackage{cite}
\usepackage{environ}
\usepackage{comment}

  \makeatletter
  \newcommand{\numberlike}[2]{%
     \expandafter\def\csname c@#1\endcsname{%
         \expandafter\csname c@#2\endcsname}%
  }
  \makeatother

  \def\DefaultNumberTheoremWithin{section}

  \theoremstyle{plain}
  \newtheorem{lemma}{Lemma}
     \numberwithin{lemma}{\DefaultNumberTheoremWithin}
     \labelformat{lemma}{Lemma~#1}
  
     \numberwithin{claim}{\DefaultNumberTheoremWithin}
     \numberlike{claim}{lemma}
     \labelformat{claim}{Claim~#1}
  \newtheorem{theorem}{Theorem}
     \numberwithin{theorem}{\DefaultNumberTheoremWithin}
     \numberlike{theorem}{lemma}
     \labelformat{theorem}{Theorem~#1}
  \newtheorem{corollary}{Corollary}
     \numberwithin{corollary}{\DefaultNumberTheoremWithin}
     \numberlike{corollary}{lemma}
     \labelformat{corollary}{Corollary~#1}
     
  \newtheorem{proposition}{Proposition}
     \numberwithin{proposition}{\DefaultNumberTheoremWithin}
     \numberlike{proposition}{lemma}
     \labelformat{proposition}{Proposition~#1}

     \numberwithin{conjecture}{\DefaultNumberTheoremWithin}
     \numberlike{conjecture}{lemma}
     \labelformat{conjecture}{Conjecture~#1}

  \newtheorem*{warning*}{Warning}

  \theoremstyle{definition}
  \newtheorem*{definition}{Definition}

  \theoremstyle{definition}
  
     \numberwithin{question}{\DefaultNumberTheoremWithin}
     \numberlike{question}{lemma}
     \labelformat{question}{Question~#1}

  \theoremstyle{definition}
  
     \numberwithin{problem}{\DefaultNumberTheoremWithin}
     \numberlike{problem}{lemma}
     \labelformat{problem}{Problem~#1}

  \theoremstyle{remark}
  \newtheorem{remark}{Remark}
     \numberwithin{remark}{\DefaultNumberTheoremWithin}
     \numberlike{remark}{lemma}
     \labelformat{remark}{Remark~#1}
  \theoremstyle{remark}

  \newtheorem{example}{Example}
     \numberwithin{example}{\DefaultNumberTheoremWithin}
     \numberlike{example}{lemma}
     \labelformat{example}{Example~#1}

\labelformat{equation}{#1}
\labelformat{figure}{Figure~#1}
\labelformat{table}{Table~#1}
\labelformat{chapter}{Chapter~#1}
\labelformat{appendix}{Appendix~#1}
\labelformat{section}{Section~#1}
\labelformat{subsection}{Subsection~#1}

\def\bA{\mathbf A}
\def\bC{\mathbf C}
\def\bD{\mathbf D}

\def\Fun{\mathbf{Fun}}
\def\Set{\mathbf{Set}}
\def\blt{*}
\def\one{[1]}

\def\kk{{R}}
\def\ve{\varepsilon}

\def\id{\textnormal{id}}
\def\letbe{\coloneqq}
\def \la{\longrightarrow}
\def \lm{\longmapsto}
\def\for{\textnormal{for }}
\def\sgn{\textnormal{sgn}}
\def\Dsym{\Delta^{\textnormal{sym}}_a}
\def\ZZ{\mathbb Z}

\def\sym{\textnormal{sym}}
\def\rev{\textnormal{r}}
\def\trp{\textnormal{t}}
\def\pirev{q}

\def\symtowermap{\varphi}
\def\revtowermap{\psi}

\def\bd{\boldsymbol d}
\def\bs{\boldsymbol s}
\def\bt{\boldsymbol t}
\def\br{\boldsymbol r}
\def\bg{\boldsymbol\gamma}
\def\bConnection{\bg}

\def\Deg{\mathrm{D}}    
\def\sDeg{\Deg^{\mathrm{sym}}}  
\newcommand{\DpsD}[2]{\big(\Deg+\sDeg\big)_{#1} #2} 

\def\cDeg{\mathrm{cD}}     

\def\Con{\mathrm{\Gamma}}  
\def\tCon{\Con^{\trp}}   
\def\rCon{\Con^{\rev}}   
\def\rtCon{\Con^{\rev,\trp}} 
\def\ACon{\Con^A} 
\def\posCon{\Con^+} 

\def\Connection{\gamma}

\def\fd{\mathfrak{d}}
\def\fs{\mathfrak{s}}
\def\fConnection{\mathfrak{g}}
\def\ft{\mathfrak{t}}
\def\fr{\mathfrak{r}}
\newcommand{\del}{\partial}
\newcommand{\dblar}[3]{\ar[yshift=3.1pt]{#1}{#2}
	\ar[yshift=-3.1pt]{#1}[swap]{#3}} 

\newcommand{\Mod}[1]{\mathbf{Mod}(#1)}
\newcommand{\Ch}[1]{\mathbf{Ch}(#1)} 
\newcommand{\sMod}[2]{\mathbf{sMod}_\textnormal{a}^{\textnormal{#1}}(#2)}  
\newcommand{\cMod}[2]{\mathbf{cMod}^{\textnormal{#1}}(#2)}  
\newcommand{\sSet}[1]{\mathbf{sSet}_a^\textnormal{#1}} 
\newcommand{\cSet}[1]{\mathbf{cSet}^\textnormal{#1}} 
\newcommand{\SimpMoore}[1]{N_{#1}^\Delta} 
\newcommand{\CubMoore}[1]{N_{#1}^\square} 
\newcommand{\Sym}[1]{S_{[#1]}}
\newcommand{\argdel}[2]{\del^{#1}_{#2}}
\newcommand{\argdelkub}[2]{\del^{#1}_{#2}}
\newcommand{\pmap}[2]{p_{#2}^{#1}}
\newcommand{\prev}[2]{q_{#2}^{#1}}
\newcommand{\ptr}[1]{u_{#1}}
\def\pitr{u}
\newcommand{\firstappearence}[1]{{\sf #1}}
\newcommand{\hrev}[1]{h^r_{#1}}
\newcommand{\repcub}[3]{\square^{#1,#2}([1]^{#3})} 
\newcommand{\repcat}[2]{\square^{#1,#2}} 
\newcommand{\projzero}[1]{\overline{(#1)}}
\newcommand{\projone}[2]{\overline{(#1,#2)}}
\newcommand{\projtwo}[3]{\overline{(#1,#2,#3)}}
\newcommand{\projthree}[4]{\overline{(#1,#2,#3,#4)}}

\newcounter{cubicalcounter}[section] 
\makeatletter
\newenvironment{cubical}
{%
  \refstepcounter{cubicalcounter}
  \start@align\@ne\st@rredtrue\m@ne
  \tag{$\square_{\thecubicalcounter}$}
}{%
  \endalign
}
\makeatother

\newcounter{simplicialcounter}[section] 
\makeatletter
\newenvironment{simplicial}
{%
  \refstepcounter{simplicialcounter}
  \start@align\@ne\st@rredtrue\m@ne
  \tag{$\triangle_{\thesimplicialcounter}$}
}{%
  \endalign
}
\makeatother

\DeclareFontFamily{U}{dmjhira}{}
\DeclareFontShape{U}{dmjhira}{m}{n}{ <-> dmjhira }{}
\DeclareRobustCommand{\yo}{\textnormal{\usefont{U}{dmjhira}{m}{n}\symbol{"48}}}
  
\author{Curtis Greene}
\address{ Haverford College, Haverford, PA 19041, USA}
\email{cgreene@haverford.edu}
\author{Volkmar Welker}
\address{Fachbereich Mathematik und Informatik, Philipps-Universit\"at Marburg, 35032 Marburg, Germany}
\email{welker@mathematik.uni-marburg.de}
\author{Georg Wille}
\address{Fachbereich Mathematik und Informatik, Philipps-Universit\"at Marburg, 35032
Marburg, Germany}
\email{georg.wille@uni-marburg.de}

\title{On the homology of simplicial and cubical sets with symmetries}

\begin{document}

\begin{abstract}
    We study the homology of simplicial and cubical sets with symmetries. These are simplicial and cubical sets with additional maps expressing the symmetries of simplices and cubes. We consider the chain complex computing the homology groups in either case. We show for coefficients in fields of characteristic $0$ that the sub-complex generated by degeneracies (simplicial case) or connections (cubical case) together with all $x - \sgn(t)tx$ for symmetries $t$ and chains $x$ is acyclic. In particular, it follows that quotienting by this sub-complex yields a chain complex with isomorphic homology. The latter leads to structural insight and a speedup in explicit computations. 
    We also exhibit examples which show that acyclicity does not hold for general coefficient rings $R$. 
\end{abstract}

\subjclass{05E99, 55U15}

\keywords{Simplicial set, cubical set, symmetry, homology}

\maketitle

\section{Introduction} \label{sec:intro}

In this paper we consider simplicial and cubical sets with symmetries
in the sense of \cite[p.58--59]{GodementBook} and
\cite{Grandis,Grandis09}. For a commutative ring $R$ the 
simplicial and cubical $R$-modules arising as chain groups in the chain 
complexes computing the homology of the simplicial and cubical sets 
inherit the symmetries. We study when for a group of
symmetries the submodule generated by $x - \sgn(t)tx$ for 
$x$ a simplicial or cubical chain and $t$ an element of the symmetry group is an acyclic sub-complex. \ref{mainSimp} and \ref{mainCub} answer this question
positively for fields $R$ of characteristic $0$. In addition, since taking quotients by an acyclic sub-complex leaves the homology invariant, the results can be used to speed up homology 
computations. 
In \ref{sec:counter} we provide examples which show that the conditions on $R$ in
\ref{mainSimp} and \ref{mainCub} cannot be relaxed easily. 

To put this setting into the context of 
classical results, consider the following two app\-roaches to the 
cal\-cu\-lation of simplicial homology of an abstract simplicial complex $K$ with coefficients in some com\-mutative ring $R$. 

For the first, we start with the simplicial set $\mathbf{Sym}_a(K)$ whose $n$-simplices are the
$(n+1)$-tuples $(v_0,\dots,v_n)$ of not necessarily distinct vertices of $K$ for which $\{v_0,\ldots, v_n\} \in K$ with the usual face and degeneracy maps. Clearly, the $n$-simplices are invariant under the
action of $\textnormal{Sym}(\{0,\ldots,n\})$ permuting the coordinates and the action exhibits certain compatibilities with the face and degeneracy maps. Therefore, it 
constitutes an example of a simplicial set with symmetries in the sense of this paper. The corresponding chain complex has
as the chain group in degree $n$ the free $R$-module with the $n$-simplices as its basis and the simplicial differential 
as maps. With the induced action of $\textnormal{Sym}(\{0,\ldots,n\})$ the chain complex becomes a simplicial
$R$-module with symmetries.  

The second approach to the homology of $K$ is the one taken
by most textbooks (see e.g. \cite{hatcher}). 
Choose a total order $\preceq$ on the set of vertices of the simplicial complex. 
Then consider the sub-complex of the chain complex 
associated to $\mathbf{Sym}_a(K)$ with basis the $(n+1)$-tuples $(v_0,\ldots, v_n)$ for strictly increasing $v_0 \prec \cdots \prec v_n$ and 
$\{v_0,\ldots, v_n \} \in K$. 
The fact that the chain complex and its sub-complex 
lead to isomorphic
homology groups is classical (see \cite{EilenbergSteenrodBook}). 

It is usually proved by showing that the inclusion of the sub-complex is a quasi-isomorphism of chain complexes. The cokernel of this inclusion is generated by the cosets of all $(v_0,\ldots, v_n)-\sgn(t)\,(v_{t(0)},\ldots, v_{t(n)})$ for permutations $t \in \textnormal{Sym}(\{0,\dots, n\})$ and of all those $(v_0,\ldots v_n)$ for
which there are $0 \leq i < j \leq n$ with $v_i=v_j$.
Since the inclusion is a quasi-isomorphism the cokernel is acyclic. For fields $R$ of characteristic $0$ this is a  special case
of \ref{mainSimp}. 
 A more detailed account of this case, including the simplicial set of weakly increasing chains, is given in \ref{SymOr} and in the discussion after \ref{mainSimp}. 
An analogous result for general 
 simplicial $R$-modules with symmetries has been termed a ''possible conjecture'' in the first paragraph of \cite[p. 59]{GodementBook}. The results of \ref{sec:counter} show that this conjecture fails in general.

Our initial motivation for this work comes from a concrete example of a cubical set with symmetries: Discrete homotopy theory of graphs $G$ (see \cite{HTgraphs06,ComparisonWithPath,carranza2022cubicalsettingdiscretehomotopy}). The corresponding homology theory $H^{\textnormal{disc}}_\blt(G)$ was first studied in \cite{DiscreteHomology}. Discrete homotopy theory was
originally devised as a tool to analyze social networks represented as graphs
in \cite{Atkin1,Atkin2}. For its applications it is crucial that the 
algorithmic calculation of the homotopy and homology of 
a given graph can be done as efficiently as possible. For homology, reducing the size of the chain groups is
the most effective way to achieve algorithmic speed ups.
The consequences of \ref{mainCub} have already been shown to 
lead to substantial speed ups in 
homology computations (see \cite{NathanAndChris}). 
\ref{ex:cube} provides a detailed description of this case.

The paper is organized as follows. In 
\ref{sec:basic} the simplicial 
(in \ref{subsec:simpcat}) and cubical
(in \ref{subsec:cubcat}) sets and modules and
their respective homology theories are introduced. The main
results \ref{mainSimp} and \ref{mainCub} are stated
together with a detailed description of the examples outlined in \ref{sec:intro}.
In \ref{sec:simcub} a functor from cubical to simplicial sets (both with symmetries) is constructed. Beyond its use in \ref{sec:homology} it is shown to provide
means for connecting the cubical and simplicial Moore complex and for proving 
acyclicity of the 
positive connection sub-complex for cubical sets.
In \ref{sec:homology} we then provide the proofs of \ref{mainSimp} and \ref{mainCub}.
Finally in \ref{sec:counter} we 
construct examples
which show that the assumptions on $R$ in the two main results cannot easily be relaxed. Thereby we 
also provide a counterexample to the conjecture suggested
\cite[p. 59]{GodementBook}. In \hyperref[appendix]{Appendix~A}
we collect identities for the 
maps in simplicial and cubical sets with symmetries. The identities 
involving face and degeneracy maps are standard and can be found
in many places in the literature (see e.g. \cite[Chapter 8.1]{WeibelBook} or \cite{Kan}).
In the simplicial case the identities involving symmetries can be found in \cite{Grandis} (in different notation). In the cubical case most of the identities can be found
in \cite{GrandisMauri}. They are all simple consequences of the
definitions of the simplex and cube categories.

\section{Basic concepts and statement of main results} \label{sec:basic} 

For an integer $n \geq -1$ we write $[n]$ for the set $\{0,1,\dots n\}$ with the convention $[-1]=\emptyset$.
We denote by $\Sym{n}$ the group of permutations of $[n]$, which 
is up to shifting the elements by $1$ the usual symmetric group $S_{n+1}$ on $n+1$ letters.
For two categories $\bC$ and $\bD$ we write 
$\Fun(\bC,\bD)$ for the functor category with objects the 
covariant functors from
$\bC$ to $\bD$ and morphisms the natural transformations. We denote by 
$\bC^{ op}$ the opposite category of $\bC$ and by $\Set$ the category of sets. All rings will be commutative and unital. We will denote by ${\bf Mod}(R)$ the category of modules over a ring $R$. For an Abelian category $\bA$ we denote by $\Ch{\bA}$ the category of chain complexes in $\bA$.

\subsection{Simplex Categories}\label{subsec:simpcat}

The \firstappearence{augmented simplex category} $\Delta_a$ is the category with objects $[n]$ for integers $n \geq -1$ and
morphisms generated by

\medskip

\begin{enumerate}
    \item[(fac)] \firstappearence{face maps} $\bd_i$ for $0\leq i \leq n $
    
    \medskip
    
    \begin{align*}
        \bd_i\colon \left\{ \begin{array}{ccc} [n-1]&\la & [n]    \\
        j& \lm & \begin{cases}
            j & \for j <i \\
            j+1& \textnormal{otherwise}
              \end{cases} \end{array} \right.
    \end{align*}

    \medskip
    
    \item[(deg)] and \firstappearence{degeneracy maps} $\bs_i$ for $0\leq i \leq n$

    \medskip
    
        \begin{align*}
        \bs_i\colon\left\{ \begin{array}{ccc} [n+1]&\la& [n]    \\
        j&\lm &\begin{cases}
            j & j \leq i\\
            j-1& \textnormal{otherwise}.
               \end{cases}\end{array} \right.
    \end{align*}
\end{enumerate}

Furthermore, the \firstappearence{augmented simplex category with symmetries}  $\Delta_a^ {sym}$ is the
category whose objects are the
objects of the augmented simplex category $\Delta_a$ and whose maps are generated by face and degeneracy maps as well as

\begin{enumerate}
    \item[(trp)] simple transpositions $t_i$ for $0\leq i\leq n-1$

        \begin{align*}
        \bt_i\colon \left\{ \begin{array}{ccc} [n]&\la &[n]    \\
        j&\lm &\begin{cases}
            i+1& \for j=i \\
            i& \for j=i+1 \\
            k& \textnormal{otherwise}.
        \end{cases} \end{array} \right.
    \end{align*}
\end{enumerate}

\bigskip

We call a map in $\Dsym$ which is a composition of simple transpositions a \firstappearence{symmetry}. Clearly, for a fixed $n$ the set of
symmetries can be identified with the symmetric group 
$\Sym{n+1} \cong \Sym{[n]}$. As usual, we define the sign \firstappearence{sign} $\sgn(\bt)$ of a symmetry $\bt=\bt_{i_1}\cdots \bt_{i_k}$ to be $1$ if $k$ is even and to $-1$ if $k$ is odd.
It is a simple exercise (see e.g., \cite{Grandis}) 
 that every map between the finite sets $[m]$ and $[n]$ can be written as a composition of these $\bd_i,\bs_i$ and $\bt_i$, i.e. $\Dsym$ is a skeleton of the category of finite sets.

\begin{definition}~
\begin{enumerate}
    \item The \firstappearence{category of symmetric simplicial sets with symmetries} $\sSet{sym}$ is the functor category $$\mathbf{Fun}\big(\,(\Dsym)^{op},\mathbf{Set}\,\big).$$ 
    We also refer to symmetric simplicial sets as \firstappearence{simplicial sets with symmetries}.
    \item The \firstappearence{category of symmetric simplicial $R$-mod\-ules} $\sMod{sym}{R}$ for a commutative ring $R$ is the functor category $$\mathbf{Fun}((\Dsym)^{op},\mathbf{Mod}(R)).$$ We also refer to symmetric simplicial $R$-mod\-ules as \firstappearence{simplicial $R$-modules with sym\-me\-tries}.
\end{enumerate}

Analogously, the categories $\sSet{}$ of simplicial sets and 
$\sMod{}{R}$ of simplicial $R$-modules are the categories of functors
from $\Delta_a$ to $\mathbf{Set}$ or $\mathbf{Mod}(R)$ respectively.

A (symmetric) simplicial $\mathbb{Z}$-module will be called a \firstappearence{(symmetric) simplical abelian group}.
\end{definition}

\medskip
 The \firstappearence{unaugumented (symmetric) simplex category} is the full subcategory of the augumented (symmetric) simplex category with all objects except $[-1]$. An unagumented (symmetric) simplicial $R$-module is a $\Mod{R}$ valued presheaf over the unaugumented (symmetric) simplex category. 
 From an unaugumented simplicial $R$-module $X$ one can obtain an au\-gu\-men\-ted (symmetric) simplical $R$-Module by defining $X[-1]=0$. This way the category of unaugumented simplicial $R$-modules can be embedded into the category of augumented simplical $R$-modules.

For us, the categories $\sSet{sym}$, $\sMod{sym}{R}$, 
$\sSet{}$ and
$\sMod{}{R}$
are all augmented, but for the sake of a more succinct notation we 
have dropped "augmented" from their names. The subscript $a$ still indicates
that the categories are augmented. We write $\sSet{(sym)}$ and $\sMod{(sym)}{R}$ when we wish to leave open whether the simplicial sets or simplicial \(R\)-modules have symmetries or not.

Postcomposing with the free functor $R(\cdot)\colon\mathbf{Set}\to \mathbf{Mod}(R)$ gives rise to a functor from the category of (symmetric) simplicial sets to the category of (symmetric) simplicial $R$-modules. Note that not every (symmetric) simplicial $R$-module arises this way.

\medskip
Throughout this paper we will denote
the maps in $\Dsym$ by bold faced letters and their
images in a functor category 
$\mathbf{Fun}\big(\,(\Dsym)^{op},\blt\,\big)$
by the respective non-bold faced letter (e.g., 
the image of $\bd_i$ denoted as $d_i$).

\begin{example} \label{SymOr}
    Given a finite simplicial complex $K$ consider the following two ways to construct a simplicial set from $K$.
    \begin{enumerate}
        
        \item The first simplicial set has $n$-simplices defined by
        \begin{align*}
            \mathbf{Sym}_a(K)[n] \letbe \big\{(v_{i_0}, v_{i_1}, \dots , v_{i_n})~\big\vert~\{v_{i_0},v_{i_1},\dots ,v_{i_n}\}\in K\big\}.
        \end{align*}
        Now equipped with the obvious face and degeneracy maps as well as symmetries acting by the permutations of the tuple entries $\mathbf{Sym}_a(K)$ is a symmetric simplicial set. 
        \item For the second example choose a total order $\{\,v_0 \prec v_1 \prec \cdots  \prec v_m\, \}$ on the vertex set of $K$. Then for an 
        integer $n \geq -1$
        set
    \begin{align*}
    \mathbf{Or}_a(K)[n]\letbe \big\{(v_{i_0}\preceq v_{i_1}\preceq \cdots \preceq v_{i_n})~\big\vert~\{v_{i_0},v_{i_1},\dots ,v_{i_n}\}\in K\big\}
    \end{align*}
        equipped with the obvious face and degeneracy maps is a simplicial set. In general, this simplicial set cannot be endowed with symmetries.
    \end{enumerate}
    Notice there is a canonical inclusion $j\colon\mathbf{Or}_a(K)\hookrightarrow \mathbf{Sym}_a(K)$ of simplicial sets.
\end{example}

\begin{definition}[alternating face map complex]\label{def:alternatingfacemapcomplex}
    The alternating face map complex $$C_\blt X \colon
    \cdots \rightarrow C_{n}X \xrightarrow{\partial_n} C_{n-1} X \xrightarrow{\partial_{n-1}} \cdots \xrightarrow{\partial_0} C_0 X
    \xrightarrow{\partial_{-1}} C_{-1} X \rightarrow 0$$ of a (symmetric) simplicial $R$-module $X\in \sMod{(sym)}{R}$ is given by $C_nX=X([n])$ and with differential
    \[
    \partial_n =\sum_{i=0}^n (-1)^i d_i.
    \]
\end{definition}

The usual calculation shows that $\partial_{n-1} \circ \partial_n = 0$, the 
construction gives rise to a functor 
$C_\blt\colon\sMod{(sym)}{R} \to \mathbf{Ch}(\mathbf{Mod}(R))$
into the category $\mathbf{Ch}\big(\,\mathbf{Mod}(R)\,\big)$ of chain complexes of $R$-modules. 
For a (symmetric)
simplicial set $X$ we write $C_\blt(X;R)$
for the chain complex which is the image of $X$ under 
$$\sSet{(sym)} \overset{R(\cdot)}{\la} \sMod{(sym)}{R} \overset{C_\blt}{\la} \mathbf{Ch}(\mathbf{Mod}(R)).$$
We can then postcompose with the homology functor 
$H_\blt$ 
to obtain homology $R$-modules. If $X$ is a (symmetric) simplicial set we
write $H_\blt\big(\,X;R\,\big)$ for $H_\blt \big(\,C_\blt (X;R)\,\big)$ and if $X$ is a (symmetric) simplicial module we write $H_\blt\big(\,X\big)$ 
for $H_\blt \big(\,C_\blt X\,\big)$.

If $X$ is a symmetric simplicial set, the $R$-module $C_n(X;R)$ 
inherits the symmetries and hence the action of the symmetric group,

but a priori it is not clear how and if the symmetries still
act on $H_\blt\big(\, X;R\,\big)$.
For example, let $K$ be the abstract simplicial complex
$2^{ \{1,2\}}$. Cal\-cu\-lat\-ing the differential of $\sigma \letbe (1,1,2)-(1,1,1)
\in C_2(\,\mathbf{Sym}_a(K);R\,)$ yields $0$ but
\[\partial_1 (t_1\,\sigma) = (2,1)+(1,2)-2 (1,1).\] In particular, the 
kernel of $\partial_n$ is not necessarily invariant under
symmetries. 

\begin{definition}~
    \begin{itemize}
        \item[($\Deg$)]Let $X$ be a (symmetric) simplicial $R$-module. The \firstappearence{degeneracy sub-complex} $\Deg_\blt X$ is the sub-complex of $C_\blt X$ generated by the images of the degeneracy maps.
        \item[($\sDeg$)] Let $X$ be a symmetric simplicial $R$-module. The \firstappearence{symmetry sub-complex} $\sDeg_\blt X$ is the sub-complex of $C_\blt X$ generated by chains of the form $x-\sgn(t)tx$ for a symmetry $t$ and a chain $x\in C_n X$
    \end{itemize}
\end{definition}
The fact that $\Deg_\blt X$ is indeed a sub-complex, i.e. $\del_n\big(\, \Deg_n X \,\big)\subseteq \Deg_{n-1} X$ follows immediately from \eqref{eq:III}, for $\sDeg_\blt X$ this is will be shown in \ref{SymSplit}. While simple examples show that for a symmetric simplicial module $\Deg_\blt X$ is not necessarily invariant 
under symmetries, this clearly holds for $\sDeg_\blt X$. For a field $R$ of characteristic different from two or more generally a ring $R$ in which $2$ is a unit, it follows from \eqref{eq:V} that a degeneracy $s_ix$ in $\sDeg_\blt X$ can be written as
\[
s_ix=\frac{1}{2}\big(\,s_ix + t_is_ix\,\big),
\]
and therefore $\Deg_\blt X$ is contained in $\sDeg_\blt X$ in this case. For arbitrary coefficients however, this is no longer true; in particular, the chain complex $\DpsD{\blt}{X}\letbe \Deg_\blt X + \sDeg_\blt X$ can be larger than $\sDeg_\blt X$.

It is well known that the degeneracy sub-complex is acyclic for any simplicial set (see \cite[Proof of Theorem 8.3.8]{WeibelBook}). The sub-complex $\DpsD{\blt}{X}$ is known to be acyclic for $X= R(\mathbf{Sym}_a(K))$, where $K$ is a simplicial complex (see \cite[p. 173-176]{EilenbergSteenrodBook}). It has been suggested as a possible conjecture that the sub-complex $\Deg_\blt X + \sDeg_\blt X$ is acyclic for general simplicial $R$-modules \cite[first paragraph, p.59]{GodementBook}. 
Our first main result, proved in \ref{sec:homology}, states that this conjecture has a positive answer if $R$ is a field of characteristic $0$.

\begin{theorem}\label{mainSimp}
    Let $R$ be a field of characteristic $0$ and $X\in\sMod{\sym}{R}$ a symmetric simplicial $R$-module, 
    then $\sDeg_\blt X$ is acyclic.

    In particular, $$H_*(X) \cong H_*\big(\,C_*(X)/\sDeg_* (X)\,\big)$$ and
    for a homology class $[\alpha] \in H_n\big(\,C_*(X)/\sDeg_* (X)\,\big)$
    and a symmetry $t$ we have $$[t \alpha] = \sgn(t) [\alpha].$$ 
\end{theorem}

 For arbitrary coefficients however, neither the symmetry sub-complex $\sDeg_\blt X$ nor the sum of sub-chain complexes $\Deg_\blt X + \sDeg_\blt X$ are necessarily acyclic, even for the symmetric simplicial $R$-modules arising from simplicial sets. In \ref{sec:counter} we construct a symmetric simplicial set, for which the symmetry sub-complex and $\Deg_\blt X + \sDeg_\blt X$ 
 both are not acyclic when $R = \ZZ$.

Note that for an arbitrary ring $R$ and a simplicial complex $K$, the simplicial sets $\mathbf{Or}_a(K)$ and $\mathbf{Sym}_a(K)$ satisfy
$$
\faktor{C_\blt(\mathbf{Sym}_a(K);R)}{\DpsD{\blt}{(\mathbf{Sym}_a(K);R)}} \simeq \faktor{C_\blt(\mathbf{Or}_a(K);R)}{\Deg_\blt(\mathbf{Or}_a(K);R)} 
$$
 which is a chain complex with one generator in degree $n$ for every (non-degenerate) $n$-simplex of $K$. The isomorphism is given by the embedding of $\mathbf{Or}_a(K)$ into $\mathbf{Sym}_aK$. Of course, this is the chain complex usually used to define reduced simplicial homology $\widetilde{H}^{\Delta}$ with coefficients in $R$ of the simplicial complex $K$ (see e.g., \cite{hatcher}). It is also the second complex calculating simplicial homology discussed in the introduction. From the classical results in  \cite{EilenbergSteenrodBook} showing acyclicity of $\sDeg_\blt\big(\,\mathbf{Sym}_a(K);R\,\big)$ and $\Deg_\blt\big(\,\mathbf{Or}_a(K);R\,\big)$ it then follows that
\[
\widetilde{H}^{\Delta}_\blt\big(\,K;R\,\big)\simeq H_\blt\big(\,\mathbf{Or}_a(K);R\,\big)\simeq H_\blt\big(\,\mathbf{Sym}_a(K);R\,\big).
\]

\subsection{Cube Categories} \label{subsec:cubcat}
The cube category $\square$ (with connections) is the category with objects $[1]^n$, $n \geq 0$, and morphisms generated by:
\begin{enumerate}
    \item[(fac)] face maps $\bd_i^\ve$ for $1\leq i \leq n$
    \begin{align*}
        \bd_i^\ve\colon \left\{ \begin{array}{ccc} \one^{n-1}&\la &\one^n  \\
        (v_1,\ldots, v_{n-1})&\lm &(v_1,\ldots,v_{i-1},\varepsilon,\dots v_{n-1})
        \end{array} \right.
    \end{align*}
    \item[(deg)] degeneracies $\bs_i$ for $1\leq i \leq n+1$
    \begin{align*}
        \bs_i\colon \left\{ \begin{array}{ccc} \one^{n+1}&\la &\one^{n}  \\
        (v_1,\ldots,v_{n+1})&\lm &(v_1,\ldots ,v_{i-1},v_{i+1},\ldots,v_{n+1}) \end{array} \right.
    \end{align*}
    \item[(con)] connections $\bg_i^\ve$ for $1\leq i \leq n$
    \begin{align*}
        \bg_i^\ve\colon \left\{ \begin{array}{ccc} \one^{n+1}&\la&\one^{n}  \\
        (v_1,\ldots, v_{n+1})&\lm & (v_1,\ldots,v_{i-1},m^\ve(v_i,v_{i+1}),\ldots v_{n+1})\end{array} \right.
    \end{align*}
    where $m^\ve$ is the maximum (minimum) for $\ve=0$ ($\ve=1$).
   \end{enumerate}

Consider the following sets of maps between the objects of the cube category.

 \begin{enumerate}
    \item[(trp)] Simple transpositions $\bt_i$ for $1\leq i \leq n-1$ 
    \begin{align*}
        \bt_i\colon \left\{ \begin{array}{ccc} \one^n&\la& \one^n  \\
        (v_1,\ldots, v_{n})&\lm &(v_1,\ldots, v_{i-1},v_{i+1},v_{i},\ldots ,v_{n})\end{array} \right.
    \end{align*}
    \item[(rev)] Simple reversals $\br_i$ for $1\leq i \leq n$
    \begin{align*}
        \br_i\colon \left\{ \begin{array}{ccc} \one^n & \la & \one^n\\
        (v_1,\ldots, v_{n})&\lm & (v_1,\ldots ,v_{i-1},1-v_{i},v_{i+1},\ldots, v_{n}) \end{array} \right.
    \end{align*}
\end{enumerate}

 A composition of simple transpositions is called an order preserving symmetry, since these are precisely those symmetries, which preserve the component-wise order on $[1]^n=\{0<1\}^n$. A composition of simple reversals is called a reversal. For a subset $\emptyset \neq A\subset \{\trp,\rev\}$ the cube category with $A$-symmetries $\square^A$ is the category whose objects are the
objects of the cube category and whose maps are generated by face maps, degeneracies, connections and the maps from the sets of maps specified in
$A$. When we wish to remain agnostic with respect to the choice of $A$, we use the term $A$\firstappearence{-symmetry} to refer to an order-preserving symmetry if $A=\{\trp\}$, a reversal if $A=\{\rev\}$, or a composition of reversals and order-preserving symmetries if $A=\{\rev,\trp\}$. As usual the \firstappearence{sign} $\sgn(t)$ of an $A$-symmetry $t$ is $1$ if it can be written as a composition of an even number of simple transpositions and simple reversals, and $-1$ otherwise.
 In the Appendix one finds a list of identities \ref{Ceq:I}--\ref{Ceq:XVII} which hold in cube categories.

\begin{definition}
Let $A\subset \{\trp,\rev\}$.
\begin{enumerate}
  \item The \firstappearence{category of $A$-symmetric cubical sets} $\cSet{A}$ is the functor category $$\mathbf{Fun}\big(\,(\square^A)^{op},\mathbf{Set}\,\big).$$
  We also refer to $A$-symmetric cubical sets as \firstappearence{cubical sets with $A$-symmetries}.
    \item For a commutative ring $R$, the \firstappearence{category of $A$-symmetric cubical
    $R$-modules} $\mathbf{cMod}^A(R)$ is the functor category 
     $$\mathbf{Fun}\big(\,(\square^A)^{op},\Mod{R}\,\big)$$
     We also refer to $A$-symmetric $R$-modules as \firstappearence{cubical $R$-modules with $A$-symmetries}.
  
    \end{enumerate}
    If $A=\emptyset$ we omit the superscript.
    An ($A$-symmetric) cubical $\mathbb{Z}$-module will be called an \firstappearence{($A$-symmetric) cubical abelian group}.
\end{definition}

As in the simplicial case, postcomposing with the free functor $R(\cdot)\colon\Set\to \Mod{R}$ induces a functor $R(\cdot)\colon\cSet{A}\to \cMod{A}{R}$.

For a cubical $R$-module $X$ and $n \geq 1$
we write $\cDeg_n(X)$ for the submodule of $X[1]^n$ generated by the images of the degeneracy maps $s_i$.

\begin{definition}[cubical alternating face map complex]
 The alternating face map complex 
 $$C_\blt X \colon
    \cdots \rightarrow C_{n} X \xrightarrow{\partial_n} C_{n-1} X \xrightarrow{\partial_{n-1}} \cdots \xrightarrow{\partial_0} C_0 X
    \xrightarrow{\partial_{-1}} C_{-1} X \rightarrow 0$$ of an $A$-symmetric cubical $R$-module $X\in \mathbf{cMod}(R)^{A}$ is given by $$C_0 X = X([1]) \text{ and }C_nX = \faktor{X([1]^n)}{\cDeg_n(X)},$$
    where 
    for $n \geq 1$ 
    \[
    \partial_n = \sum_{i=0}^n(-1)^i (d_i^0-d_i^1).
    \]
\end{definition}

It follows from the cubical identity \ref{Ceq:II} that $\del_n\big(\,\cDeg_n(X)\,\big)\subseteq \cDeg_{n-1}(X)$, thus the maps $\del_n\colon C_n(X)\to C_{n-1}(X)$ are well defined. The usual calculation shows that $\partial_{n-1} \circ \partial_n = 0$ and 
hence also in the cubical case this
construction gives rise to a functor 
$$C_\blt\colon\cMod{A}{R} \to \Ch{\Mod{R}}$$
into the category $\Ch{\Mod{R}}$ of chain complexes of
$R$-modules. We use $C_\blt$ for both the simplicial alternating face map complex (see \ref{def:alternatingfacemapcomplex}) and the cubical alternating face map complex. The nature of the input will determine which version is meant. The next lemma shows that symmetries still act on $C_\blt X$.

\begin{lemma}
   Let $X$ be an $A$-symmetric cubical $R$-module. Then for all $n$ the sub-$R$-module 
   $\cDeg_n(X)$ of $X[1]^n$ is invariant under $A$-symmetries. In particular, there is an action of the 
   $A$-symmetries on $C_nX$.
\end{lemma} 
\begin{proof}
    This follows immediately from the identities \ref{Ceq:IX} and \ref{Ceq:XIV}.
\end{proof}

As in the simplicial case we can also postcompose with the homology functor 
$H_\blt$ 
and obtain homology $R$-modules $H_\blt \big(\,C_\blt X\,\big)$. If $X$ is a cubical set we
write $H_\blt\big(\,X;R\,\big)$ 
for $H_\blt \big(\,C_\blt R(X) \big)$
and if $X$ is a cubical module we write $H_\blt\big(\,X\big)$ 
for $H_\blt \big(\,C_\blt X\,\big)$.

\begin{definition}~
\begin{itemize}
        \item[($\Con$)]  Let $X$ be a cubical $R$-module. The \firstappearence{connection sub-complex} $\Con_\blt X$ is the sub-complex of $C_\blt X$ generated by the images of the connection maps.
    \item[($\ACon$)] Let $X$ be a cubical $R$-module with $A$-symmetries with $\emptyset \neq A \subseteq\{\rev,\trp\}$. The \firstappearence{$A$-symmetry sub-complex} $\ACon_\blt X$ is the sub-complex of $C_\blt X$ generated by
         chains of the form $x-\sgn(t)tx$ for an $A$-symmetry $t$ and a chain $x\in C_n X$
        The $A$-symmetry sub-complex will be denoted by $\rCon_\blt X$ for $A=\{\rev\}$, by $\tCon_\blt X$ for $A=\{\trp\}$ and by $\rtCon_\blt X$ for $A=\{ \rev,\trp \}$. 
        \end{itemize}
\end{definition}
For a cubical set $Y$, The connection sub-complex of $R(Y)$ has been shown to be acyclic in \cite{HomologyConnections} (see also \cite{CKT},\cite{DohertyConnections}).
Similarly to the simplicial setting, our second main result concerns the $A$-symmetry sub-complex.  

\begin{theorem} \label{mainCub}
    Let $R$ be a field of characteristic $0$ and $X\in\cMod{A}{R}$ a cubical $R$-module with $A$-symmetries where $\emptyset \neq A \subseteq\{\rev,\trp\}$,
    then $\ACon_\blt X$ is acyclic.

    In particular, $$H_*(X) \cong H_*\big(\,C_*(X)/\ACon_* (X)\,\big)$$ and
    for a homology class $[\alpha] \in H_n\big(\,C_*(X)/\ACon_* (X)\,\big)$
    and an $A$-symmetry $t$ we have $$[t \alpha] = \sgn(t) [\alpha].$$ 
     If $A=\{\rev\}$ the condition on the characteristic of the field can be relaxed to $\textnormal{char}(R)\neq 2$.
\end{theorem}
The proof will be given in \ref{sec:homology} for each value of $A$ separately.
Again, the corresponding result for arbitrary coefficient rings is false; in \ref{sec:counter} we will give examples for cubical sets with $A$-symmetries, where the $A$-symmetry sub-complex with integer coefficients is not acyclic. 

The following example demonstrates an application of \ref{mainCub}. 
 
\begin{example} \label{ex:cube} By a graph we mean a simple undirected graph $G=(V_G,E_G)$, with vertices $V_G$ and edges $E_G\subset \binom{V_G}{2}$. For two graphs $G = (V_G,E_G)$ and $H
=(V_H,E_H)$ a graph homomorphism $f$ is a map $f\colon V_G\to V_H$ such that for every edge $\{\,g_1,g_2\,\}\in E_G$ either $f(g_1)=f(g_2)$ or $\{\,f(g_1),f(g_2)\,\}\in E_H$. The cube graphs $I^n$ are the graphs with vertices $V(I^n)=\{0,1\}^n$ connected by an edge if and only they differ in exactly one coordinate. For a graph $G$, consider the set of graph homomorphisms 
\[
\textnormal{Hom}_{\mathbf{Graph}}(I^n, G)
\]
from $I^ n$ to $G$.
We obtain a cubical set $N_1G$ (using the notation from
\cite{carranza2022cubicalsettingdiscretehomotopy}) by setting
$N_1G([1]^n) = \textnormal{Hom}_{\mathbf{Graph}}(I^n, G)$ and maps given by precomposing $f \in \textnormal{Hom}_{\mathbf{Graph}}(I^n, G)$ 
with the graph maps
    \begin{align*}
        \fd_i^\ve\colon \left\{ \begin{array}{ccc} I^{n-1}&\la &I^n  \\
        (v_1\dots v_{n-1})&\lm &(v_1\dots,v_{i-1},\varepsilon,\dots v_{n-1})
        \end{array} \right.
    \end{align*}
for $1 \leq i \leq n$ to yield the face maps $d_i^ \epsilon$, with the maps
\begin{align*}
        \fs_i\colon \left\{ \begin{array}{ccc} I^{n+1}&\la &I^{n}  \\
        (v_1\dots,v_{n+1})&\lm &(v_1\dots v_{i-1},v_{i+1}\dots,v_{n+1}) \end{array} \right.
    \end{align*}
for $1 \leq i \leq n+1$ to yield the degeneracy maps $s_i$ and with the maps
\begin{align*}
        \fConnection_i^\ve\colon \left\{ \begin{array}{ccc} I^{n+1}&\la&I^{n}  \\
        (v_1\dots, v_{n+1})&\lm & (v_1\dots,v_{i-1},m^\ve(v_i,v_{i+1})\dots v_{n+1})\end{array} \right.
    \end{align*}
to yield the connections $\Connection_i^\ve$. 
 
 The application of the homology
functor $H_*$ to the image of the composition 
\[
\mathbf{Graph}\overset{N_1}{\la} \mathbf{cSet}\overset{R(\cdot)}{\la} \mathbf{cMod}(R)\overset{C}{\la} \mathbf{Ch}(\mathbf{Mod}(R))
\]
defines a well studied discrete homology theory of graphs with coefficients in a ring $R$ (see e.g., \cite{DiscreteHomology,ComparisonWithPath}). Observe that $N_1G$ can be endowed with 
simple transpositions $t_i$ for $1\leq i \leq n-1$ 
    \begin{align*}
        \ft_i\colon \left\{ \begin{array}{ccc} \one^n&\la& \one^n  \\
        (v_1\dots v_{n})&\lm &(v_1\dots v_{i-1},v_{i+1},v_{i},\dots v_{n})\end{array} \right.
    \end{align*}
    and simple reversals $r_i$ for $1\leq i \leq n$
    \begin{align*}
        \fr_i\colon \left\{ \begin{array}{ccc} \one^n & \la & \one^n\\
        (v_1\dots v_{n})&\lm & (v_1\dots v_{i-1},1-v_{i},v_{i+1}\dots v_{n}) \end{array} \right.
    \end{align*}
to obtain a cubical set with symmetries $t_i$  and reversals $r_i$. In particular, the assumptions of
\ref{mainCub} are satisfied.
 
The calculation of 
homological graph invariants is of interest in 
topological data analysis (see the papers in the collection \cite{CS24}). Classically
this meant analyzing the clique complex of a graph $K(G)$ or similar simplicial complexes associated to graphs. Similiarly to the cubical construction above, the sets of graph maps from the complete graphs to $G$ form a simplicial set, which is easily seen to be isomorphic to $\mathbf{Sym}_aK(G)$. More recently, other homology theories, like the cubical one just introduced,
have come into the focus (see for example\cite{NathanChrisTDA}, \cite{CS24} or
\cite{CM18}).
These applications are often confined to small
graphs,  because the ranks of the chain groups explode and make homology computation inefficient. 
It has been demonstrated in \cite{NathanAndChris} that a significant speedup in the computation of discrete cubical homology with rational coefficients can be achieved as a consequence of \ref{mainCub}.
\end{example}

\section{A functor relating cubical and simplicial homology} \label{sec:simcub}
Let $R$ be a unital commutative ring.
In \ref{subsec:FunConstruction} a functor from the category of cubical $R$-modules $\cMod{}{R}$ to the category of simplicial $R$-modules $\sMod{}{R}$ is constructed. This functor enables us to use our main result in the simplicial case, \ref{mainSimp}, in the proof of the main result for cubes with symmetries \ref{mainCub}. 
We demonstrate in \ref{subsec:FunApplication}, that this functor is useful beyond its concrete purpose in the proof of \ref{mainCub}. 
Using the functor we provide a new and remarkably simple proof of the acyclicity of the positive connection sub-complex for cubical $R$-modules.

\subsection{The functor}\label{subsec:FunConstruction} 

The description of the functor will occasionally require the use of the cubical and simplicial structure maps in the same computation. In the rest of the paper the two cases are strictly separated and
we are able to use the same letters for the structure maps in either case. To avoid confusion in this section we add a superscript $\Delta$ to the simplicial structure maps (e.g., $s_i^\Delta$ refers to a simplicial degeneracy, whereas $s_i$ refers to a cubical one).

Next, we construct the functor. For that,
recall that for $X \in \cMod{}{R}$ the $R$-module $\cDeg_n(X)$ is the submodule of $X(\,\one^{n}\,)$ generated by the images of the cubical degeneracies. 
As a first step, we associate to $X \in \cMod{}{R}$ a family of $R$-modules $SX$ and show in \ref{lem:sx} that it is a simplicial $R$-module.

For $[n]\in \textnormal{Ob}(\Delta)$ set \[
    SX([n])\letbe \faktor{X(\,\one^{n+1}\,)}{\cDeg_{n+1}(X)}.
    \]
    
\begin{lemma} \label{lem:welldefined}
    Let $X \in \cMod{}{R}$ then for $n \geq 0$
    and $0 \leq i \leq n-1$ we have
    \begin{itemize}
    \item[(i)] $(d_{i+1}^ 0-d_{i+1}^ 1) \cDeg_n(X) \subseteq \cDeg_{n-1}(X)$,
    \item[(ii)] $\Connection_{i+1} \cDeg_n(X)
    \subseteq \cDeg_{n+1}(X)$,
    \item[(iii)] if $X \in \cMod{t}{R}$ then $t_{i+1} \cDeg_n(X) \subseteq
    \cDeg_n(X)$.
    \end{itemize}
    In particular, 
    $(d_{i+1}^ 0-d_{i+1}^ 1)$,
    $\Connection_{i+1}$ and $t_{i+1}$ 
    induce well defined maps  
    {\small
    $$(d_{i+1}^ 0-d_{i+1}^ 1) \colon SX([n+1])
    \rightarrow SX([n]), 
    \Connection_{i+1} \colon SX([n-1]) \rightarrow SX([n]), t_{i+1} \colon SX([n])
    \rightarrow SX([n]).$$}
\end{lemma}
\begin{proof}
    It follows from \eqref{Ceq:II} that the
    image of a degeneracy under $(d_{i+1}^ 0-d_{i+1}^ 1)$ is either $0$ or the image of
    a degeneracy map. In either case it is
    contained in $\cDeg_{n-1}(X)$. 
    
    By \eqref{Ceq:VI} the image of a degeneracy under $\Connection_{i+1}$ is again the
    image of a degeneracy map.
    Finally, \eqref{Ceq:IX} shows that 
    the image under $t_{i+1}$ of a degeneracy under $\Connection_{i+1}$ is again the
    image of a degeneracy map.
    The remaining claim now follows from the definition of $SX$.
\end{proof}
    
By \ref{lem:welldefined} for $n \geq 0$
    and $0 \leq i \leq n-1$ 
the following maps are well defined:
    \begin{align*}
   d_i^\Delta\colon \left\{ \begin{array}{ccc} SX([n+1])&\longrightarrow &SX([n])\\
    x&\longmapsto &(\,d_{i+1}^0-d_{i+1}^1\,)x\end{array} , \right.
    \end{align*}

    \begin{align*}s_i^\Delta\colon \left\{ \begin{array}{ccc} SX([n-1])&\longrightarrow &SX([n])\\
    x&\longmapsto &\Connection_{i+1}^0x \end{array} . \right.
    \end{align*}
    
    If $X \in \cMod{t}{R}$, then 
    we also define
    \begin{align*}
    t_i^\Delta\colon \left\{ \begin{array}{ccc} SX([n])&\la & SX([n])\\
    x&\lm& t_{i+1} x \end{array} . \right.
    \end{align*}
    
\begin{lemma} \label{lem:sx}
Let $X \in \cMod{}{R}$.
Then $SX$ with the degeneracy maps $d_i^\Delta$ and face maps $s_i^\Delta$ 
 is a simplicial $R$-module.
 
 If $X \in \cMod{\trp}{R}$, then
 $SX$ with symmetries $t_i^ \Delta$ is a simplicial $R$-module with symmetries. 
\end{lemma}
\begin{proof}
   First, it is easily checked that the range of $i$ for which $d_i^ \Delta,s_i^ \Delta$ and $t_i^ \Delta$ are defined matches the requirements for
   (symmetric) simplicial $R$-modules.

   By \cite[Proposition 8.1.3]{WeibelBook}
   to show that $SX$ is a simplicial $R$-module, it suffices to verify \eqref{eq:I}, \eqref{eq:II} and \eqref{eq:III} 
   for $d_i^ \Delta$ and $s_i^ {\Delta}$.
      \begin{itemize}
        \item[\eqref{eq:I}] Let $j\leq i$, then:
        \[
    d^\Delta_id^\Delta_j = (d_{i+1}^0-d_{i+1}^1)(d_{j+1}^0-d_{j+1}^1)\overset{\eqref{Ceq:I}}{=} (d_{j+2}^0-d_{j+2}^1)(d_{i+1}^0-d_{i+1}^1)=d^\Delta_{j+1}d^\Delta_i
    \]
        \item[\eqref{eq:II}]Let again $j\geq i$:
    \[
    s_i^\Delta s_j^\Delta =\Connection^0_{i+1}\Connection^0_{j+1}\overset{\eqref{Ceq:IV}}{=}\Connection^0_{j+2}\Connection^0_{i+1}= s_{j+1}^\Delta  s_i^\Delta
    \]
        \item[\eqref{eq:III}]
        \begin{align*}
        d_i^\Delta s_j^\Delta=(d_{i+1}^0-d_{i+1}^1)\Connection_{j+1}^0 &\overset{\eqref{Ceq:V}}{=}\begin{cases}
 \Connection_{j+1}^0(d_{i+1}^0-d_{i+1}^1) & \textnormal{for } j < i - 1 \\
\textnormal{id}-s_{i+1}id_{i+1}^1 & \textnormal{for } j = i - 1, i \\
\Connection_{j}^0d_{i+1}^\ve  & \textnormal{for } j > i
\end{cases}\\& = \begin{cases}
    s_{j}^\Delta d_{i-1}^\Delta &\for j<i\\
    \id &\for j=i,i-1\\
    s_{j-1}^\Delta d_i^\Delta &\for j>i
\end{cases}
        \end{align*}
        Here we use, that 
        $s_{i+1} d_{i+1}^1=0$ since cubical degeneracies are set to $0$.
        \end{itemize}
        Thus $SX$ is a simplicial $R$-module.

   By \cite[Theorem 4.2]{Grandis} to show that $SX$ is a symmetric simplicial $R$-module, it suffices to verify in addition \eqref{eq:IV}-\eqref{eq:VIII}
   for $d_i^ \Delta$, $s_i^ {\Delta}$ and $t_i^ \Delta$.
  \begin{itemize}
 
        \item[\eqref{eq:IV}]
        \begin{align*}
        d_i^\Delta t_j^\Delta =(d_{i+1}^0-d_{i+1}^1)t_{j+1}&\overset{\eqref{Ceq:VIII}}{=} \begin{cases}
    t_{j+1} (d_{i+1}^0-d_{i+1}^1) & \for j<i-1\\
    (d_{i}^0-d_{i}^1) & \for j=i-1\\
    (d_{i+2}^0-d_{i+2}^1) & \for j=i\\
    t_{j} (d_{i+1}^0-d_{i+1}^1) & \for j>i
\end{cases}\\
&=\begin{cases}
    d_i^\Delta t_j^\Delta  & \for j<i-1 \\
    d_{i-1}^\Delta  & \for j=i-1\\
    d_{i}^\Delta  & \for j=i \\
    t_{j-1}^\Delta d_i^\Delta  & \for j>i
\end{cases}
    \end{align*}
    \item[\eqref{eq:V}]
    \begin{align*}
    t_i^\Delta s_j^\Delta =t_{i+1}\Connection_{j+1}^0 &\overset{\eqref{Ceq:VII}}{=}\begin{cases}
    \Connection_{j+1}^0 t_{i} & \for j<i-1\\
    t_{i+1}\Connection_{i+2}^0t_{i+1}& \for j=i-1\\
    \Connection_{j+1}^0 & \for j=i\\
    t_{i+2}\Connection_{i+1}^0t_{i+1}& \for j=i+1\\
    \Connection_{j+1}^0 t_{i+1}& \for j>i+1
\end{cases}\\&=\begin{cases}
 s_j^\Delta  t_{i-1}^\Delta  & \for j< i-1\\ 
 t_i^\Delta s_{i+1}^\Delta t_i^\Delta &\for j=i-1\\
 s_j^\Delta  &\for j=i\\
 t_{i+1}^\Delta s_i^\Delta t_i^\Delta &\for j=i+1 \\
  s_j^\Delta  t_i^\Delta  & \for j>i+1 
\end{cases} \\
\end{align*}
    \item[\eqref{eq:VI}--\eqref{eq:VIII}] follow immediately from \eqref{Ceq:X}--\eqref{Ceq:XII}
    \end{itemize}
    \end{proof}

    Now we can state the main result of this section.
    
\begin{theorem}\label{FunctorThm}
    The assignment $S\colon\cMod{}{R} \to \sMod{}{R}$ is a functor that extends to a functor $S\colon\cMod{\trp}{R} \to \sMod{\sym}{R}$.
\end{theorem}
\begin{proof}
 Let $F\colon X\to Y$ be a morphism of cubical $R$-modules, then in every degree $n$ it holds that $F_n(\cDeg_n(X))\subseteq \cDeg_n(Y)$, therefore $F$ passes down to well defined family of maps $SF=\{SF_n\colon SX_n\to SY_n\}$. The compatibility of $F$ with the cubical face maps, connections (and  transpositions) implies the compatibility of $SF$ with the simplicial degeneracies, boundaries (and symmetries).
\end{proof}

\begin{remark}
    Alternatively, one can set $s_i^\Delta x=-\Connection_{i+1}^0 x$ and also obtain a functor, using essentially the same proof.
\end{remark}

The next result demonstrates the use of \ref{FunctorThm}.

\begin{corollary}\label{ComplexAndSCorollary}
    For any $X\in \cMod{}{R}$, there is an equality of chain complexes
    \[
    C_\blt X = C_{\blt-1}SX .
    \]

\end{corollary}
Recall that we denote by $C_\blt$ both the cubical and the simplical alternating face map complex. Since $X$ is a cubical $R$-module, $C_\blt X$ is the cubical alternating face map complex of $X$, whereas $C_\blt SX$ is the simplical alternating face map complex of the simplicial $R$-module $SX$.
\begin{proof}
    In every degree $n$ we obtain by definition
    \[
    C_{n-1}SX = SX[n-1] = \faktor{X(\,\one^{n}\,)}{\cDeg_n(X)} = C_nX.
    \]
    It remains to be shown that the differentials agree, for this, we denote by $\del^{SX}_*$ the differential of $C_{*}SX$ and by $\del^X_*$ the differential of $C_*X$. 
    \begin{align*}
    \del^{SX}_{n-1}&=\sum_{i=0}^{n-1}(-1)^{i} d_{k}^\Delta
    =\sum_{i=0}^{n-1}(-1)^{i} (d_{i+1}^0-d_{i+1}^1)
    =\sum_{i=1}^{n}(-1)^{i-1} (d_{i}^0-d_{i}^1)
    =\del^X_n
    \end{align*}
\end{proof}

\subsection{Application to the Moore complex}\label{subsec:FunApplication}
In the last part of this section we provide a result, which is 
independent of the main goals of the paper. 
We have included it because it illustrates the functor $S$ and because we consider it interesting in itself.

In the result we compare the Moore complexes of simplicial and
cubical Abelian groups. In the simplicial case Moore complexes are a well known construction
(see \cite[Definition 8.3.6]{WeibelBook}).
Only recently, the cubical analog was defined in \cite{CKT}. We will demonstrate that the
simplicial and cubical Moore complex are related by the functor $S$ from \ref{FunctorThm}. 

\begin{definition}[Moore Complex]~

    \begin{itemize}
        \item The \firstappearence{simplicial Moore complex} of a simplicial Abelian group $X$ is the chain complex  $$\SimpMoore{\blt} X \colon
    \cdots \rightarrow \SimpMoore{k}X \xrightarrow{\partial_k} \SimpMoore{k-1} X \xrightarrow{\partial_{k-1}} \cdots \xrightarrow{\partial_0} \SimpMoore{0} X
    \xrightarrow{\partial_{-1}} \SimpMoore{-1} X \rightarrow 0$$  given by $$\SimpMoore{k}X= \faktor{X[k]}{\Deg_k(X)}$$ 
    and 
    \[
    \partial_k = \sum_{i=0}^k(-1)^i d_i.
    \]
        \item The \firstappearence{cubical Moore complex} of a cubical Abelian group $X$ is the chain complex  $$\CubMoore{\blt} X \colon
    \cdots \rightarrow \CubMoore{k}X \xrightarrow{\partial_k} \CubMoore{k-1} X \xrightarrow{\partial_{k-1}} \cdots \xrightarrow{\partial_0} \CubMoore{0} X
     \rightarrow 0$$  with chain groups $$\CubMoore{k}X= \faktor{X(\one^k)}{B_k (X)},$$ 
    where $B_k (X)$ is the sub-group of 
    $X(\one^k)$ generated by the images of the cubical degeneracies and the positive connections
    \[
    \partial_k = \sum_{i=1}^k(-1)^{i-1} (d_i^0-d_i^1).
    \]
    \end{itemize}
    \end{definition} 
    
    Note that in \cite[Definition 8.3.6]{WeibelBook} the simplicial Moore complex is defined 
    with chain groups $\bigcap_{i=0}^{k-1}\ker(d_i)$ in degree $k$ and differentials $d_k$. It is well known that this constructions yields a chain complex
    isomorphic to the Moore complex according to our definition (see e.g., \cite[Theorem III.2.1]{JardineGoerss}). The above definition is more convenient for our purposes. 
    For some, possibly empty $A\subseteq \{\trp,\rev \} $ and $ X\in \cMod{A}{R}$, the \firstappearence{positive connection sub-complex} $\posCon_*X$ is the sub-complex of $C_\blt X$ generated by positive connections.
     This has been shown to be acyclic for cubical Abelian groups arising from cubical sets in \cite{HomologyConnections}.
    
    \begin{corollary}\label{DegUndConComplex}
      Let $X \in \cMod{A}{R}$. The positive connection sub-complex $\posCon_*X$ is iso\-mor\-phic to the simplicial degeneracy sub-complex of $SX$. In  particular, it holds that $$\CubMoore{\blt}X=\SimpMoore{\blt-1}SX.$$
    \end{corollary}
    \begin{proof}
    In every degree $n$, it holds that $C_n X=SX[n-1]$, which is the quotient of $X(\one^n)$ by the cubical degeneracies. The simplicial degeneracies of $SX$ are, by definition of $S$, precisely the positive connections of $X$. Therefore
    we have, $$\Deg_{n-1} SX\cong \posCon_nX.$$
    Thus the simplicial Moore complex of $SX$ in degree $(n-1)$ is obtained from $X(\one^n)$ by quotienting with the sum of the cubical degeneracies and the positive connections of $X$. By definition, this is the same as the cubical Moore complex $\CubMoore{n-1}X$
    of $X$.
    Here, the fact that the differentials agree follows from \ref{ComplexAndSCorollary}.
    \end{proof}

    The following result has been shown in \cite{HomologyConnections}, but the theory developed so far allows us to give a different proof.
    
    \begin{corollary}
    For a cubical Abelian group $X$, the positive connection sub-complex $\posCon_*X$ is acyclic.
    \end{corollary}
    \begin{proof}
    By \ref{DegUndConComplex}, it holds that $\posCon_*X=\Deg_{\blt-1} SX$. Using this, the claim follows directly from the fact that the simplicial degeneracy sub-complex of a simplicial abelian group is acyclic (\cite[Proof of Theorem 8.3.8]{WeibelBook}).
    \end{proof}

\section{Homology and Symmetry} \label{sec:homology}

In this section, we provide the proofs of \ref{mainSimp} and \ref{mainCub} 
about the homology of the cubical or simplicial $R$-modules with symmetries (and/or reversals in the cubical cases), where $R$ is a field of characteristic $0$. Note, we will see that 
in the cubical case when only reversals are added the
assumption $\text{char}(R)\neq 2$ suffices.

First, in
\ref{sub:simplicial} we prove the result for simplicial $R$-modules with symmetries. From this, we obtain in \ref{sub:cubical} the corresponding result for cubical $R$-modules using the functor $S$ from \ref{sec:simcub}. In \ref{sub:reversal} we then verify the result for cubical sets with reversals. Lastly, in \ref{sub:hyperoctahedral}, we argue that the compatibility of symmetries and reversals allows us to then complete the proof for
cubical sets with symmetries and reversals, or
equivalently the full hyperoctahedral group.

\subsection{Simplicial \texorpdfstring{$R$}--Modules with Symmetries}
\label{sub:simplicial}

For the proof of the first technical result, \ref{dtLemma}, 
of this section it will be
convenient to return to the simplex categories and their maps, written in 
bold faced letters. 

As usual consider a symmetry $\bt=\bt_{i_1}\cdots \bt_{i_k}$ in 
$\Dsym \big(\,[n],[n] \,\big)$ as a permutation in $S_{[n]}$. 
We associate to $\bt$ the corresponding $(n+1)\times (n+1)$ permutation matrix $A_{\bt}$ with rows and columns indexed with the numbers from 
$0$ to $n$, i.e. $(A_{\bt})_{j,i}=\delta_{j,\bt(i)}$. 
Deleting the $i$-th column and $\bt(i)$-th row of the permutation matrix $A_{\bt}$ results in a new permutation matrix. We denote 
the permutation represented by this matrix by $\symtowermap_i(\bt)$. More precisely, $\symtowermap_i(\bt)$ is the permutation, which for $k\in [n-1]$ takes the values
\[
\symtowermap_i(\bt)(k)=\begin{cases}
    \bt(k) & \for k<i, ~\bt(k) <\bt(i)\\
    \bt(k)-1 & \for k<i,~ \bt(k)>\bt(i)\\
    \bt(k+1) & \for k\geq i,~ \bt(k+1)<\bt(i)\\
    \bt(k+1) -1 &\for k\geq i,~\bt(k+1)>\bt(i).
\end{cases}
\]

Note, that in general $\symtowermap_i$ is not a group homomorphism. Nevertheless, the following lemma shows that
it behaves well with respect to the sign.

\begin{lemma}\label{dtLemma}
     The map $\symtowermap_i\colon \Sym{n} 
     \to \Sym{n-1}$ is surjective, the cardinality of the preimage of every $\bt \in \Sym{n-1}$ is $n+1$,
     and for $\bt \in S_{[n]}$ it holds that
    \[
    \sgn(\symtowermap_i(\bt))=(-1)^{i+\bt(i)} \sgn(\bt).
    \]
    and
    \[
    \symtowermap_i(\bt^{-1})=\left(\,\symtowermap_{\bt(i)}(\bt)\,\right)^{-1}.
    \]
\end{lemma}

\begin{proof}
Given a permutation $\bt\in \Sym{n-1}$ with permutation matrix $A_{\bt}$ where rows and columns indexed by $[n-1]$. For $j \in [n]$ consider the permutation matrix $A_{\bt'}$ obtained from $A_{\bt}$ by inserting after the row indexed by $(j-1)$
and after the column indexed by $(i-1)$ a row and a column of $0$s and then replacing the position $(j,i)$ by $1$. For $j=0$ or $i=0$ we mean an
insertion before the row indexed by $0$ or column indexed by $0$.
Applying $\symtowermap_i$ to the permutation corresponding to $A_{\bt'}$ yields $\bt$. It follows that
$\symtowermap_i$ is surjective.  Since there are $n+1$ possible positions for $j$, there are at least $n+1$ permutations in the preimage of every 
$\bt \in \Sym{n-1}$. From
$(n+1) \cdot \# \Sym{n-1} = \#
S_{[n]}$ we then deduce that 
each preimage contains exactly $n+1$ elements.

The formula for the sign of $\symtowermap_i(\bt)$ follows from Laplace expansion via
\[
\sgn(\symtowermap_i(\bt))=\mathrm{det}(A_{\symtowermap_i(\bt)})=(-1)^{i+\bt(i)}\det(A_{\bt})=(-1)^{i+\bt(i)}\sgn(\bt).
\]
The last part follows from the fact that the inverse of a permutation matrix is given by transposition.
\end{proof}

Postcomposing $\symtowermap_i(\bt)$ with the face map $\bd_{\bt(i)}$ yields

\[
\bd_{\bt(i)} \symtowermap_i(\bt)(k)=\begin{cases}
\bt(k) & \for~ k<i\\
\bt(k+1) & \for~ k\geq i
\end{cases}=\bt\bd_i(k).
\]

Now we return to a simplicial $R$-module $X$ with symmetries. For $x\in X[n]$, we obtain from 
the above formula the following in terms of the maps in $X$ (which by our convention are non-bold faced).

\begin{align} \label{eq:phi}
    d_i t x & =\symtowermap_i(t)d_{t(i)} x.
\end{align}

Fix a field $R$ with characteristic $0$ and let $X\in \sMod{sym}{R}$. In every degree $n$ and for every $0\leq k\leq n$ consider the endomorphism $\pmap{k}{n}$ of $X([n])$ given by the formula
\[
\pmap{k}{n}(x)=\sum_{t\in \Sym{k}}\frac{1}{(k+1)!}\sgn(t)\,t\,x,
\]
where we identify $\Sym{k}$ with the subgroup of $\Sym{n}$ fixing $k+1,\ldots, n$ pointwise. In particular, we have $\pmap{0}{n}(x)=x$. The following identity is
immediate from the definition of $\pmap{n}{n}$:
\begin{align} \label{eq:equiv}  
  \pmap{n}{n}(t\,x) & = \sgn(t) \pmap{n}{n}(x) \text{ for } x \in X[n] \text{ and } t \in \Sym{n}. 
\end{align}

 It will be convenient to introduce the following notation for some integer $k$ 
\[
    \argdel{\leq k}{n}\letbe \sum_{i=0}^k (-1)^{i}d_i\quad \textnormal{and}\quad \argdel{>k}{n}\letbe\sum_{i=k+1}^{n}(-1)^{i}d_i,
\]
obviously satisfying $\del_n=\argdel{\leq k}{n}+\argdel{>k}{n}$.
\begin{proposition}
    For $0<k\leq n$ it holds that
    \begin{align*}
    \del_n \pmap{k}{n} = \pmap{k-1}{n-1}\argdel{\leq k}{n}+\pmap{k}{n-1}\argdel{>k}{n}.
    \end{align*}
    In particular, the family $\{\,\pmap{n}{n}\,\}$ defines a chain map $p_X:C_\blt X\to C_\blt X$.
\end{proposition}
\begin{proof}
  We treat the two summands of $\del_n \pmap{k}{n}=\argdel{\leq k}{n}\pmap{k}{n}+\argdel{>k}{n}\pmap{k}{n}$ separately. For every $x\in C_nX$ it holds that
  \begin{align*}
    \argdel{\leq k}{n}\pmap{k}{n} x&=\sum_{i=0}^k(-1)^i\sum_{t\in \Sym{k}}\frac{1}{(k+1)!}\sgn(t)d_itx\\
    &\overset{\eqref{eq:phi}}{=} \sum_{i=0}^k(-1)^i\sum_{t\in \Sym{k}}\frac{1}{(k+1)!}\sgn(t)\symtowermap_i(t)d_{t(i)}x\\
    &=\sum_{t\in \Sym{k}}\frac{1}{(k+1)!}\sgn(t)\sum_{i=0}^k(-1)^i\symtowermap_i(t)d_{t(i)}x
  \end{align*}
Instead of summing over $i$, we will now sum over $j=t(i)$ and reverse the order of summation.
\begin{align*}
    \phantom{\argdel{\leq k}{n} \pmap{k}{n}}&=\sum_{t\in \Sym{k}}\frac{1}{(k+1)!}\sgn(t)\sum_{j=0}^k(-1)^{t^{-1}(j)}\symtowermap_{t^{-1}(j)}(t)d_{j}x\\
    &=\sum_{j=0}^k\sum_{t\in \Sym{k}}\frac{1}{(k+1)!}\sgn(t)(-1)^{t^{-1}(j)}\symtowermap_{t^{-1}(j)}(t)d_{j}x
\end{align*}
Applying \ref{dtLemma} for every $j$ yields
\begin{align*}
\phantom{d_{\leq k}\pmap{k}{n}}&=\sum_{j=0}^k\sum_{u\in \Sym{k-1}}\frac{k+1}{(k+1)!}(-1)^{j}\sgn(u)ud_{j}x\\
&=\pmap{k-1}{n-1}\argdel{\leq k}{n} x
\end{align*}
The claim for the second summand follows immediately from the simplicial identity \eqref{eq:IV}, that is $d_jt=td_j$ for a symmetry $t\in \Sym{k}$ 
and $j>k$. 
\end{proof}

Let $P_\blt X$ be the image of the chain map $p_X$.

\begin{proposition}\label{SymSplit}
The kernel of $p_X$ is the symmetry sub-complex $\sDeg_\blt X$ and the image $P_\blt X$ of $p_X$ is given as 
\begin{align*}
     P_nX\letbe&\Bigg\langle\, \Bigg\{ \sum_{t \in \Sym{n}} \sgn(t) tx \,\bigg\vert\, x\in X(n)\Bigg \}\, \Bigg\rangle_{R}.
\end{align*}

The exact sequence of chain complexes
  \[
   \begin{tikzcd}
    0\arrow{r}{}&\sDeg_\blt X\arrow[hook]{r}{} &C_\blt X\arrow{r}{p_X}&P_\blt X \arrow{r}&0
    \end{tikzcd}
\]
splits. In addition, in each degree $n$ there is a direct sum decomposition of $\Sym{n}$ representations
\begin{align*}
C_nX=\sDeg_nX\oplus P_nX.
\end{align*}
\end{proposition}
\begin{proof}
We first show that $\sDeg_\blt X$ is contained in the
kernel of $p_X$.  By \eqref{eq:equiv} and linearity of $p_X$ for any $x\in C_nX$ and  $t\in \Sym{n}$ we have
 $$p_X\big(\,x-\sgn(t)tx\,\big)=0.$$ Therefore, $\sDeg_\blt X$ is contained in the kernel of $p_X$. 

Conversely, assume $p_X(x)=0$
for some $x \in X([n])$, then 
\begin{align*}
    x&=x-p_X(x)\\
    &=x-\frac{1}{(n+1)!}\sum_{t \in \Sym{n}} \sgn(t)t x\\
    &=\frac{1}{(n+1)!}\sum_{t \in \Sym{n}} \big(\,x- \sgn(t)t x\,\big),
\end{align*}
which is a linear combination of generators of 
$\sDeg_\blt X$ and hence an element of
$\sDeg_\blt X$. 

It follows that $\sDeg_\blt X$ is the kernel of $p_X$.

The claim about the image $P_nX$ of $p_X$ is immediate
from the definition of $p_X$ and the fact that $R$ is
a field of characteristic $0$. 

Since $p_X^2=p_X$, the inclusion of $P_\blt X$ into $C_\blt X$ splits the sequence.

 By \eqref{eq:equiv} the kernel and the image of $p_X$ 
    are invariant under symmetries.
    Thus, for a fixed $n$ all three chain groups are $\Sym{n}$-modules.
\end{proof}

At this point, we can define the crucial chain homotopy.

\begin{proposition}\label{PropSimpChainHtpy}
    For $n \geq -1$ set
    $$h_n \colon \left\{ \begin{array}{ccc} C_nX & \rightarrow & 
    C_{n+1} X \\ x & \mapsto &h_n(x)=\displaystyle{\sum_{k=0}^n} (-1)^k\pmap{k}{n+1} s_k x\end{array} \right. . $$
    Then we have
    $$\partial_{n+1} h_{n} + h_{n-1}\partial_n = \id_{C_nX} - \pmap{n}{n}.$$
    Thus, $h_\blt$ is a chain homotopy between
    $\id_{C_\blt X}$ and $p_X$
\end{proposition}
\begin{proof}
By straightforward calculation we have
\begin{align*}
    \del_{n+1} h_n(x)&=\sum_{k=0}^n(-1)^k \del_{n+1} \pmap{k}{n+1} s_k x.\\
\end{align*}
We start by considering the summand on the right hand side corresponding to $k=0$. Since $\pmap{0}{n}=\id_{C_n X}$ we have
\begin{align*} \del_{n+1} \pmap{0}{n+1} s_0 x &=\del_{n+1} s_0 x \\
    &=-s_0\argdel{>0}{n} x.
\end{align*}

Now consider the cases $0<k\leq n$
\[
(-1)^k \del_{n+1} \pmap{k}{n+1} s_k x=(-1)^k \pmap{k-1}{n}\argdel{\leq k}{n+1}s_kx+\pmap{k}{n}\argdel{>k}{n+1}s_kx
\]
since $k>1$, we use $\argdel{\leq k}{n+1}=\argdel{\leq k-1}{n+1}+(-1)^kd_k$ to obtain in the first summand
\begin{align*}
    \pmap{k-1}{n}\argdel{\leq k}{n+1}s_kx&=\pmap{k-1}{n}\argdel{\leq k-1}{n+1}s_kx +(-1)^k\pmap{k-1}{n}d_{k}s_kx\\
    &=\pmap{k-1}{n}s_{k-1}\argdel{\leq k-1}{n+1}x +(-1)^k\pmap{k-1}{n}x.
\end{align*}
Similarly, in the second summand, it holds that
\begin{align*}
    \pmap{k}{n}\argdel{>k}{n+1}s_kx &=(-1)^{k+1}\pmap{k}{n}d_{k+1}s_kx+ \pmap{k}{n}\argdel{>k+1}{n+1}s_kx\\
    &=(-1)^{k+1}\pmap{k}{n}x-\pmap{k}{n}s_k\argdel{>k}{n}x.
\end{align*} 
We rearrange the summands
\begin{align*}
    \del_{n+1} h_n(x)= &-s_0\argdel{>0}{n}x+\sum_{k=1}^n(-1)^k\pmap{k-1}{n}s_{k-1}\argdel{\leq k-1}{n}x  \\ &+ \pmap{k-1}{n}x-\pmap{k}{n}x+ (-1)^{k+1}\pmap{k}{n}s_k\argdel{>k}{n}x\\
    =&-\sum_{k=0}^{n-1}(-1)^k \pmap{k}{n}s_k\del_n x+\sum_{k=1}^{n}\pmap{k-1}{n}x-\pmap{k}{n} x\\
    =&-h_{n-1}(\del_n x)+x-\pmap{n}{n}(x).
\end{align*}
\end{proof}

Now we are in position to provide the 
proof of our main result on 
symmetric simplicial $R$-modules.

\begin{proof}[Proof of \ref{mainSimp}]
    By \ref{PropSimpChainHtpy} we have that $p_X$ is 
    chain homotopic to the identity. In addition, the inclusion of $P_\blt X \hookrightarrow C_\blt X$ 
    splits the sequence
    \[
    \begin{tikzcd}
    0\arrow{r}{}&\sDeg_\blt X\arrow[hook]{r}{} &C_\blt X\arrow{r}{p_X}&P_\blt X \arrow{r}&0.
    \end{tikzcd}
    \]
    It follows that $P_\blt X$ is a deformation 
    retraction of $C_\blt X$. From the long exact 
    sequence in homology we then deduce
    that $\sDeg_\blt X$ is acyclic.
    In particular, there are isomorphims:

    \[
    H_\blt (X) = H_\blt(C_\blt X) \simeq H_\blt (C_\blt X /\sDeg_\blt X ) \simeq H_\blt(P_\blt X).
    \]

    If $[\alpha] \in H_\blt (C_\blt X /\sDeg_\blt X )$ is a homology class then
    for a symmetry $t$ we have $\alpha-\sgn(t) t\alpha \in \sDeg_\blt$. Thus
    $[\alpha-\sgn(t) t \alpha] = 0$. It follows that $[\alpha] = \sgn(t)[t\alpha]$.
    
\end{proof}

\subsection{Cubical \texorpdfstring{$R$}--Modules with order preserving Symmetries} \label{sub:cubical}

Let $X \in \cMod{\trp}{R}$ be a symmetric cubical $R$-module where $R$ is a field of characteristic $0$. Similarly to the 
situation in the simplicial case we can identify the symmetries in degree $n$ with the symmetric group $S_n$.
Note that in the simplicial case the identification 
in degree $n$ was with $\Sym{n} \cong S_{n+1}$.
As in the simplicial case we consider for a given integer $n \geq 0$ the sub-module
\begin{align*}
P_n^{\trp}X\letbe&\Bigg\langle\, \Bigg\{ \,\sum_{t \in S_n} \sgn(t) t x\, \bigg\vert\, x\in X(n)\,\Bigg\} \,\Bigg\rangle_{R}.
\end{align*}

of $C_nX$. Let $p^{\trp}_X\colon C_\blt X\to P_{\blt}^{\trp}X$ be the map which is in degree $n$ given by
\begin{align*}
    p^{\trp}_X(x)=\sum_{t\in S_n}\frac{1}{n!} \sgn(t)t x.
\end{align*}
Recall that by \ref{ComplexAndSCorollary}, the functor $S\colon\cMod{\trp}{R} \to \sMod{\sym}{R}$ gives rise to the identification
\[
C_\blt SX= C_{\blt+1} X.
\]
Furthermore, the symmetry complex of $SX$ can be identified with the $\trp$-symmetry complex $\tCon_n X$ as follows.
\begin{lemma}\label{CubSimpIdent}
It holds that
\begin{align*}
    \tCon_{\blt+1} X=\sDeg_{\blt} SX,\quad P^{\trp}_{\blt+1} X=P_{\blt} SX\quad \text{ and }\quad p_{SX}=p^{\trp}_X.
\end{align*}
In particular $p_X^t$ is a chain map and both $P^{\trp}_{\blt}X$ and $\tCon_{\blt} X$ are chain sub-complexes of $C_\blt X$. Also, there is a split exact sequence of chain complexes  
\[
\begin{tikzcd}
    0\arrow{r}{}&\tCon_\blt X\arrow[hook]{r}{} &C_\blt X\arrow{r}{p^{\trp}_X}&P^{\trp}_\blt X \arrow{r}&0
\end{tikzcd}
    \] and a decomposition of $S_n$ representations
    \begin{align*}
C_n=\tCon_{n}\oplus P^{\trp}_n
\end{align*}
in each degree $n\in \mathbb{N}_0$.
\end{lemma}
\begin{proof}
The claim can be deduced directly from the definition of $S$. For the symmetry complex we see that
    \begin{align*}
        \sDeg_nSX&=\Bigg\langle\, \Bigg\{\, x -\sgn(t)tx \,\bigg\vert \,x\in SX(n),t\in \Sym{n}\,\Bigg\} \,\Bigg\rangle_R\\
        &=\Bigg\langle\, \Bigg\{\, x -\sgn(t) t x \,\bigg\vert\, x\in X(\one^{n+1}),t\in S_{n+1}\,\Bigg\}\, \Bigg\rangle_R\\
        &=\tCon_{n+1}X.
    \end{align*}
    In the same way it follows by substitution of $S$ that $P^{\trp}X=PSX $ and $p_{SX}=p^{\trp}_X$.
\end{proof}

Using \ref{CubSimpIdent} we can now give the proof
of \ref{mainCub}.

\begin{proof}[Proof of \ref{mainCub} for $A=\{\trp\}$]
        Let $X\in \cMod{\trp}{R}$ and
        $S$ the operator defined before \ref{lem:welldefined}. Using the identifications obtained in \ref{CubSimpIdent}, the claim follows directly by applying the main result on simplical $R$-modules \ref{mainSimp} to $SX$. 
\end{proof}

\subsection{Cubical \texorpdfstring{$R$}--Modules with Reversals} \label{sub:reversal}
Since reversals lack a simplicial counterpart, they require a different treatment.
Fix a cubical $R$-module $X\in \mathbf{cMod}(R)^{r}$ and let $\kk$ be a field of $\textnormal{char}(\kk)\neq 2$. The reversals in degree $n$ can be identified with the group $\ZZ_2^n$ as follows. We identify $a=(a_1,\ldots, a_n)\in\ZZ_2^n$ with the reversal
\[
r_1^{a_1}\cdots r_n^{a_n}.
\]
Clearly, the sign of a given $a\in\ZZ^n$ is $1$ if $a_1+\dots +a_n$ is even, and $-1$ otherwise. For a given $1\leq i\leq n$, let $\revtowermap_i\colon\ZZ_2^n\to \ZZ_2^{n-1}$ be the projection 
deleting the $i$-th coordinate. Given a $x\in X([1]^n)$ using \eqref{Ceq:XIII} we obtain 
\begin{align*} 
d_i^{\ve} a x & = \begin{cases}
    \revtowermap_i(a)d_i^{\ve}  x, & \textnormal{if } a_i=0\\
    \revtowermap_i(a)d_i^{1-\ve}  x & \textnormal{if } a_i=1.
\end{cases}
\end{align*}
We construct a family of endomorphisms $\prev{k}{n}$ of $X([1]^n)$ for $1\leq k\leq n$. 
\begin{align*}
    \prev{k}{n}(x)=\sum_{a\in \ZZ_2^k} \frac{\sgn(a)}{2^k}ax 
\end{align*}
with the convention $\prev{k}{n}x=x$. Note, that for $\prev{k}{n}$ to be well defined we need $\textnormal{char}(k)\neq 2$. It follows immediately that
\begin{align} \label{eq:equivrev}  
  \prev{n}{n}(\,a\,x\,) & = \sgn(a) \prev{n}{n}(x) \text{ for } x \in X(\one^n) \text{ and } a \in \mathbb{Z}_2^n. 
\end{align}
Analogously to the simplicial case we set 
$$\argdelkub{\leq k}{n}=\sum_{i=1}^k (-1)^{i-1}(d_i^0-d_i^1) \text{ and }  \argdelkub{>k}{n}=\sum_{i=k+1}^n (-1)^{i-1}(d_i^0-d_i^1).$$

\begin{proposition} \label{prop:useful}
    For $1\leq k\leq n$ and $x\in X([1]^n)$ it holds that
    \[
    \argdelkub{}{n} \prev{k}{n}x= \prev{k-1}{n-1}\argdelkub{\leq k}{n}x + \prev{k}{n-1}\argdelkub{>k}{n}x.
    \]
    In particular, the family $\pirev_X = \{\prev{n}{n}\}$ gives rise to a chain map.
\end{proposition}
\begin{proof}
By linearity $d_i^\ve$ it holds that    
    \begin{align*}
    \argdelkub{}{n} \prev{k}{n}x
    &=\sum_{i=1}^{n}(-1)^{i-1}\sum_{a\in\ZZ_2^k}\frac{\sgn(a)}{2^k} (d_i^0-d_i^1)ax.
    \end{align*}
    If $i>k$, using \eqref{Ceq:XIII} we immediately obtain
    \begin{align*}
        \sum_{a\in\ZZ_2^k}\frac{\sgn(a)}{2^k} (d_i^0-d_i^1)ax&=\sum_{a\in\ZZ_2^k}\frac{\sgn(a)}{2^k} a(d_i^0-d_i^1)x\\
        &=\prev{k}{n-1}(d_i^0-d_i^1)x
    \end{align*}
    Now, consider the case $i\leq k$. Splitting the sum and applying \eqref{Ceq:XIII} yields
    \begin{align*}
        \sum_{a\in\ZZ_2^k}\sgn(a) (d_i^0-d_i^1)ax&=\sum_{\substack{a\in\ZZ_2^k\\a_i=0}}\sgn(a) (d_i^0-d_i^1)ax+\sum_{\substack{a\in\ZZ_2^k\\a_i=1}}\sgn(a) (d_i^0-d_i^1)ax\\
        &=\sum_{\substack{a\in\ZZ_2^k\\a_i=0}}\frac{\sgn(a)}{2^k} \revtowermap_i(a)(d_i^0-d_i^1)x+\sum_{\substack{a\in\ZZ_2^k\\a_i=1}}\frac{\sgn(a)}{2^k} \revtowermap_i(a)(d_i^1-d_i^0)x\\
        &=\sum_{a\in\ZZ_2^{k-1}}\frac{\sgn(a)}{2^k} a(d_i^0-d_i^1)x+\sum_{a\in\ZZ_2^{k-1}}\frac{\sgn(a)}{2^k} a(d_i^0-d_i^1)x.
    \end{align*}
    The sign of the second summand can be deduced  from  $\sgn(a)=-\sgn(\revtowermap_i(a))$ if $a_i=1$. Overall, in this case we get 
    \[
    \sum_{a\in\ZZ_2^k}\sgn(a) (d_i^0-d_i^1)ax=\prev{k-1}{n-1}(d_i^0-d_i^1)x
    \]
    as claimed.
    \end{proof}

 Using \ref{prop:useful}
 we can set up the chain homotopy.
  \begin{proposition}\label{PropRevChainHtpy}
    For $n \geq 0$ set
    $$\hrev{n} \colon \left\{ \begin{array}{ccc} C_nX & \rightarrow & 
    C_{n+1} X \\ x & \mapsto &\hrev{n}(x)=\displaystyle{\sum_{k=1}^n(-1)^k \prev{k}{n+1}\Connection_k^0x}\end{array} \right. . $$
    Then we have
    $$\del_{n+1} \hrev{n} + \hrev{n-1}\del_n = \id_{C_nX} - \prev{n}{n}.$$
    Thus, $\hrev{\blt}$ is a chain homotopy between
    $\id_{C_\blt X}$ and $\pirev_X$.
\end{proposition}

\begin{proof}
    \begin{align*}
        \argdelkub{}{n+1} \hrev{n}(x) &= \sum_{k=1}^n(-1)^k \argdelkub{}{n+1} \prev{k}{n+1}\Connection_k^0 x \\
        &=\sum_{k=1}^n(-1)^k(\prev{k-1}{n}\argdelkub{\leq k}{n+1}\Connection_k^0x + \prev{k}{n}\argdelkub{>k}{n+1}\Connection_k^0x)
        \end{align*}
    By $d_k^0\Connection_k^0x=x=d_{k+1}^0\Connection_{k+1}^0x$ this is equal to 
        
    \begin{align*}
        &=\sum_{k=1}^n(-1)^k(\prev{k-1}{n}\argdelkub{\leq k-1}{n}\Connection_k^0x+(-1)^{k-1}\prev{k-1}{n}x +(-1)^{k}\prev{k}{n}
        x + \prev{k}{n-1}\argdelkub{>k+1}{n+1}\Connection_k^0x)
    \end{align*}
We will split the sum and perform an index shift, here we use that $\argdelkub{\leq 0}{n+1}=0$ and $\argdelkub{>n+1}{n+1}=0$, so we can drop these summands.
\begin{align*}
        &=\sum_{k=1}^{n-1}(-1)^{k+1}\prev{k}{n}\argdelkub{\leq k}{n+1}\Connection_{k+1}^0x+(-1)^{k}\prev{k}{n}\argdelkub{>k+1}{n+1}\Connection_k^0x +\sum_{k=1}^n \prev{k}{n}x-\prev{k-1}{n}x 
\end{align*}
Applying \eqref{Ceq:V} to the first sum
\begin{align*}
    \sum_{k=1}^{n-1}(-1)^{k+1}\prev{k}{n}\argdelkub{\leq k}{n+1}\Connection_{k+1}^0x+(-1)^{k}\prev{k}{n}\argdelkub{>k+1}{n+1}\Connection_k^0x
    &=\sum_{k=1}^{n-1}(-1)^{k+1}\prev{k}{n}\Connection_{k}^0\argdelkub{\leq k}{n}x+(-1)^{k+1}\prev{k}{n}\Connection_k^0\argdelkub{>k}{n}x\\
    &=-\sum_{k=1}^{n-1}(-1)^{k}\prev{k}{n}\Connection_{k}^0\del x\\
    &=-\hrev{n-1}(\argdelkub{}{n} x)
\end{align*}
The second sum is a telescope sum:

\begin{align*}
\sum_{k=1}^n \prev{k}{n}x-\prev{k-1}{n}x&=\prev{n}{n}x-x\\
\end{align*}
Overall we obtain $\hrev{n}\argdelkub{}{n+1} +\argdelkub{}{n} \hrev{n-1}=\prev{n}{n}-\id_{C_nX}$ as claimed.
\end{proof}

The proof of the following proposition is almost identical to the proof of \ref{SymSplit}.

\begin{proposition}\label{RevSplit}
The kernel of $\pirev_X$ is the $r$-symmetry sub-complex $\rCon_\blt X$ and the image $P^r_\blt X$ of $p_X$ is given as 
\begin{align*}
      P_n^{r} X\letbe&\Bigg\langle \Bigg\{ \sum_{a\in \ZZ_2^k} \sgn(a)a x \bigg\vert x\in X([1]^n)\Bigg \} \Bigg\rangle_{\kk}.
\end{align*}
The exact sequence of chain complexes
  \[
   \begin{tikzcd}
    0\arrow{r}{}&\rCon_\blt X\arrow[hook]{r}{} &C_\blt X\arrow{r}{p_X}&P^{\rev}_\blt X \arrow{r}&0
    \end{tikzcd}
\]
splits. In addition, in each degree $n$ there is a direct sum decomposition of $\mathbb{Z}_2^{n}$ representations
\begin{align*}
C_n X=\rCon_n X\oplus P^{\rev}_n X.
\end{align*}
\end{proposition}
\begin{proof}
That $\rCon_\blt X$ is contained in the
kernel of $\pirev_X$ follows immediately from \eqref{eq:equivrev} and the linearity of $\pirev_X$.
Conversely, assume $\prev{n}{n}(x)=0$
for some $x \in X(\one^n)$, then 
\begin{align*}
    x&=x-\prev{n}{n}(x)\\
    &=x-\sum_{a\in \ZZ_2^n} \frac{\sgn(a)}{2^n}ax\\
    &=\sum_{a\in \ZZ_2^n} \frac{\sgn(a)}{2^n}\big(\,x- \sgn(a)a x\,\big),
\end{align*}
which is a linear combination of generators of 
$\rCon_\blt X$.
It follows that $\rCon_\blt X$ is the kernel of $\pirev_X$.

The claim about the image $P_nX$ of $p_X$ is immediate.

Since $(\pirev_X)^2=\pirev_X$, the inclusion of $P^r_\blt X$ into $C_\blt X$ splits the sequence.
 By \eqref{eq:equivrev} the kernel and the image of $\pirev_X$ 
    are invariant under symmetries.
    Thus, for a fixed $n$ all three chain groups are $\mathbb{Z}_2^n$-modules.
\end{proof}

With this, the proof of our main result in the cubical setting with reversals is completely analogous to the one in the simplicial setting.  

\begin{proof}[Proof of \ref{mainCub} for $A=\{\rev\}$]
    By \ref{PropRevChainHtpy} we have that $\pirev_X$ is 
    chain homotopic to the identity and the inclusion of $P^r_\blt X \hookrightarrow C_\blt X$ 
    splits the sequence
    \[
    \begin{tikzcd}
    0\arrow{r}{}&\rCon_\blt X\arrow[hook]{r}{} &C_\blt X\arrow{r}{\pirev_X}&P^{\rev}_\blt X \arrow{r}&0.
    \end{tikzcd}
    \]
    It follows that $P^{\rev}_\blt X$ is a deformation 
    retraction of $C_\blt X$. From the long exact 
    sequence in homology we then deduce
    that $\rCon_\blt X$ is acyclic.
    In particular, there are isomorphims:

    \[
    H_\blt (X) = H_\blt(C_\blt X) \simeq H_\blt (C_\blt X /\rCon_\blt X ) \simeq H_\blt(P^\rev_\blt X).
    \]

    If $[\beta] \in H_\blt (C_\blt X /\rCon_\blt X )$ is a homology class, then
    for a reversal $a\in \mathbb{Z}_2^n$ we have $\beta-\sgn(a) a\beta \in \rCon_\blt$. Thus
    $[\beta-\sgn(a) a \beta] = 0$. It follows that $[\beta] = \sgn(a)[a\beta]$.
    
\end{proof}
\subsection{Cubical modules with hyperoctahedral symmetry}\label{sub:hyperoctahedral}
Again, let $R$ be a field with $\textnormal{char}(R)= 0$. We show that given a cubical module with reversals and order preserving symmetries $X\in \mathbf{cMod}(R)^{\trp,\rev}$ it is also possible to quotient $C_\blt X$ by both symmetries and reversals without changing the homology. 
In a given degree $n$, the group of automorphisms of $[1]^n$ in $\square^{\rev,\trp}$ can be identified with the hyperoctahedral group $H_n\simeq S_n \ltimes\ZZ_2^n$, thus every $h\in H_n$ can be uniquely written as $h=t a$, for a symmetry $t\in S_n$ and a reversal $a \in \ZZ_2^n$.  
In a similar way to before, consider the family of maps
$\pitr_X=\{\ptr{n}\colon C_n X \to C_n X\}$ which is given in degree $n$ by
\[
\ptr{n}(x)=\sum_{h\in H_n}\frac{1}{n!2^n}\sgn(h)hx.
\]
Observe that it holds that
\[
\pitr_X=\pirev_X p^{\trp}_X=p^{\trp}_X\pirev_X.
\]
Therefore, $\pitr$ is a composition of chain maps and hence also a chain map. Also it holds that
\begin{align} \label{eq:equivboth}  
  \ptr{n}(\,h\,x\,) & = \sgn(h) \ptr{n}(x) \text{ for } x \in X(\one^n) \text{ and } h \in H_n. 
\end{align}
\begin{proposition}\label{BothSplit}
The kernel of $\pitr_X$ is the $\{\rev,\trp\}$-symmetry sub-complex $\rtCon_\blt X$ and the image $P^{\rev,\trp}_\blt X$ of $p_X$ is given as 
\begin{align*}
      P_n^{\rev,\trp} X\letbe&\Bigg\langle \Bigg\{ \sum_{h\in H_n}\frac{1}{n!2^n}\sgn(h)hx \bigg\vert x\in X([1]^n)\Bigg \} \Bigg\rangle_{\kk}.
\end{align*}

The exact sequence of chain complexes
  \[
   \begin{tikzcd}
    0\arrow{r}{}&\rtCon_\blt X\arrow[hook]{r}{} &C_\blt X\arrow{r}{\pitr_X}&P^{\rev,\trp}_\blt X \arrow{r}&0
    \end{tikzcd}
\]
splits. In addition, in each degree $n$ there is a direct sum decomposition of $H_{n}$ representations
\begin{align*}
C_n X=\rtCon_n X\oplus P^{\rev,\trp}_n X.
\end{align*}
\end{proposition}
\begin{proof}
    Using \ref{BothSplit}, the proof becomes completely analogous to the proofs of \ref{SymSplit} and \ref{RevSplit}.
\end{proof}
\begin{proof}[Proof of \ref{mainCub} for $A=\{\rev,\trp\}$]
    By \ref{PropRevChainHtpy} and \ref{PropSimpChainHtpy} we have that $\pirev_X$ and $p^{\trp}_X$ are  
    chain homotopic to the identity, therefore the same holds for their composition $\pitr_X$. Since the inclusion of $P^r_\blt X \hookrightarrow C_\blt X$ 
    splits the sequence
    \[
    \begin{tikzcd}
    0\arrow{r}{}&\rtCon_\blt X\arrow[hook]{r}{} &C_\blt X\arrow{r}{\pitr_X}&P^{\rev,\trp}_\blt X \arrow{r}&0,
    \end{tikzcd}
    \]
    it follows that $P^{\rev,\trp}_\blt X$ is a deformation 
    retraction of $C_\blt X$. From the long exact 
    sequence in homology we then deduce
    that $\rtCon_\blt X$ is acyclic.
    In particular, there are isomorphims:

    \[
    H_\blt (X) = H_\blt(C_\blt X) \simeq H_\blt (C_\blt X /\rtCon_\blt X ) \simeq H_\blt(P^{\rev,\trp}_\blt X).
    \]

    If $[\alpha] \in H_\blt (C_\blt X /\rtCon_\blt X )$ is a homology class then
    for a hyperoctahedral symmetry $h\in H_n$ we have $\alpha-\sgn(h) h\alpha \in \rtCon_\blt$. Thus
    $[\alpha-\sgn(h) h\alpha] = 0$. It follows that $[\alpha] = \sgn(h)[h\alpha]$.
    
\end{proof}

\section{Counterexamples for arbitrary coefficient rings} \label{sec:counter}

The following examples show that the 
the condition on the characteristic of the field of coefficients in \ref{mainSimp}
and \ref{mainCub} cannot be dropped
completely. For this, we provide counterexamples in integer coefficients.

\subsection{Symmetric Simplicial Abelian Groups} \label{subsec:ssag}

First, we give a combinatorial description of the simplicial 
set $X$, for which the simplicial Abelian group $\mathbb{Z}X$ will form the counterexample in the simplicial case.

For every integer $n \geq 0$, we set $X[n]$ to be the equivalence classes of $(n+1)$-tuples $(\ve_0,\ldots,\ve_{n})$ with $\ve_i\in\{0,1\}$ with respect to
the equivalence relation generated by $$(\ve_0,\ldots,\ve_{n}) \sim (1-\ve_0,\ldots,1-\ve_{n}).$$ The equivalence classes 
will be denoted $\overline{(\ve_0,\cdots ,\ve_{n})}$. The set of 
equivalence classes carries the structure of a simplicial set 
when endowed with the following maps: 
\begin{itemize}
\item face maps $d_i$ omitting the $i$-th coordinate,
\item degeneracy maps $s_i$  duplicating the $i$-th coordintate
\item and symmetries permuting the entries. 
\end{itemize}
Listing the $X[i]$ for small $i$ we obtain the following:
\begin{align*}
    X[-1] &=\emptyset  \\
    X[0]  &= \left\{\,\projzero{0}\,\right\}  \\
    X[1]  &= \left\{\,\projone{0}{0},\projone{0}{1}\,\right\}\\
    X[2]  &= \left\{\,\projtwo{0}{0}{0},\projtwo{0}{0}{1},\projtwo{0}{1}{0},\projtwo{0}{1}{1} \,\right\}\\
    X[3]  &= \Big\{\,\projthree{0}{0}{0}{0},\projthree{0}{0}{0}{1},\projthree{0}{0}{1}{0},\projthree{0}{0}{1}{1},\\ &\phantom{=\big\{\,} \projthree{0}{1}{0}{0},\projthree{0}{1}{0}{1},\projthree{0}{1}{1}{0},\projthree{0}{1}{1}{1}\,\Big\}.
\end{align*}
The simplicial Abelian group $\ZZ X$ in degree $n$ is the free Abelian group generated by the elements of $X[n]$, hence the chains of the chain complex $C_\blt X$ are $\ZZ$-linear combinations of the elements of $X[n]$. Consider the chain
\[
a=\projthree{0}{1}{1}{0}+t_2\left(\projthree{0}{1}{1}{0}\right)= \projthree{0}{1}{1}{0}+\projthree{0}{1}{0}{1},
\]
which is contained in the symmetry sub-complex $\sDeg_\blt \ZZ X$.
By calculating the boundary of $a$ we obtain
\begin{align*}
    \del\left(\projthree{0}{1}{1}{0}+\projthree{0}{1}{0}{1}\right)&=\projtwo{1}{1}{0}+\projtwo{1}{0}{1}-\projtwo{0}{1}{0}-\projtwo{0}{0}{1}\\&\quad+\projtwo{0}{1}{0}+\projtwo{0}{1}{1}-\projtwo{0}{1}{1}-\projtwo{0}{1}{0}\\
    &=0,
\end{align*}
thus $a$ is a cycle. However, it is easy to see that for any chain in $\ZZ X[4]$ the chain $\projthree{0}{1}{0}{1}$ appears an even number of times in its boundary. 
Thus the cycle $a$ is not a boundary in $C_\blt\ZZ X$ and in particular not in the sub-complex $\sDeg_\blt\ZZ X$. Therefore, we have that both $H_3\big(\,\sDeg_\blt \ZZ X\,\big)\neq 0$
and $H_3\big(\,(\Deg_\blt + \sDeg_\blt) \ZZ X\,\big)\neq 0$.

In the cubical case we will provide
in \ref{subsec:scag} a less combinatorial description of the counterexample. As a preparation we describe the above
constructed example for the simplicial case in a similar fashion.
Let $\yo: \Dsym\to \mathbf{sSet}^{sym}_a$ denote the Yoneda embedding. The transposition $t_0\in \textnormal{Hom}_{\Dsym}([1],[1])$ gives rise to an endomorphism $(\bt_0)_*\letbe \yo(\bt_0)$ of $\yo([1])$ by postcomposition.
Let
\[
\begin{tikzcd}
X\letbe\textnormal{colim}\Big(\yo([1])\dblar{r}{(\bt_0)_*}{\id}&\yo([1])\Big).   
\end{tikzcd}
\]
Since the category of sets is cocomplete, so is the category of presheaves over the symmetric simplex category. Thus this coequalizer exists and can be calculated pointwise. 
Application of the free Abelian group functor $\mathbb{Z}(\cdot)\colon\mathbf{sSet}^{sym}_a\to \mathbf{sMod}(\ZZ)^{\textnormal{sym}}_a$ yields a symmetric Abelian group $\ZZ X$,
which is easily seen to coincide with the example constructed at the beginning of this subsection.

\subsection{Symmetric cubical Abelian groups} \label{subsec:scag}

Next, we construct examples with non-acyclic symmetry sub-complexes in the cubical setting. For $A\subset \{\rev,\trp\}$ consider the Yoneda em\-bed\-ding 
\[
\yo^{A}\colon\square^{A}\to \cSet{A}
\]and for $n \in \mathbb{N}$ denote by $\square^{A,n}$ the standard cubical $n$-cube, that is
\[
\square^{A,n}=\yo^{A}([1]^n).
\]

We define for each $\emptyset \neq A$ a cubical set. 
In the definition, for $\yo^A$ we write $\yo^\trp$ if $A = \{\trp\}$, 
$\yo^\rev$ if $A = \{\rev\}$ and
$\yo^{\rev,\trp}$ if $A = \{\rev,\trp\}$.

\medskip

\begin{enumerate}
    \item $ A = \{\trp\}:$ cubical sets with order preserving symmetries

    \noindent Set $$(\bt_1)_*\letbe \yo^{\trp}(\bt_1)\in \textnormal{End}(\repcat{\trp}{2})$$ 
    and
    \[
    \begin{tikzcd}
        Y_\trp\letbe\textnormal{colim}\Big(\repcat{\trp}{2}\dblar{r}{(\bt_1)_*}{\id}&\repcat{\trp}{2}\Big).   
    \end{tikzcd}
    \]
    \item ${ A = \{\rev\}}:$ cubical sets with reversals

    \noindent Set $$(\br_1)_*=\yo^{\rev}(\br_1)\in \textnormal{End}(\repcat{\rev}{1})$$
    and 
    \[
    \begin{tikzcd}
        Y_\rev\letbe\textnormal{colim}\Big(\repcat{\rev}{1}\dblar{r}{(\br_1)_*}{\id}&\repcat{\rev}{1}\Big).   
    \end{tikzcd}
    \]
    \item ${ A = \{\rev,\trp\}:}$ cubical sets with full hyperoctahedral symmetry

    \noindent Set $$(\br_1)_*=\yo^{\rev,\trp}(\br_1)\in \textnormal{End}(\repcat{\rev,\trp}{1})$$
    and
    \[
    \begin{tikzcd}
        Y_{\rev,\trp}\letbe\textnormal{colim}\Big(\repcat{\rev,\trp}{1}\dblar{r}{(\br_1)_*}{\id}&\repcat{\rev,\trp}{1}\Big).   
    \end{tikzcd}
    \]
\end{enumerate}
 
To the $A$-symmetric cubical sets $Y_\trp,Y_\rev,Y_{\rev,\trp}$, apply the free functor $\mathbb{Z}(\cdot)\colon\cSet{A}\to \cMod{A}{\ZZ}$ to obtain an $A$-symmetric cubical Abelian groups $\ZZ Y_\trp,\ZZ Y_\rev$ and $\ZZ Y_{\rev,\trp}$ respectively. The following table shows the first four homology groups of the $A$-symmetry sub-complexes $\tCon_\blt \ZZ Y_\trp$, $\rCon_\blt\ZZ Y_\rev$ and
$\rtCon_\blt\ZZ Y_{\rev,\trp}$. 

\begin{center}
\def\arraystretch{1.3}
\begin{tabular}{c|c c c c} 
 &$H_1$ & $H_2$ & $H_3$ & $H_4$ \\ 
 \hline
 $\tCon_\blt \ZZ Y_{\trp}$& $0$&$0$&$0$&$\ZZ_2$\\
 $\rCon_\blt \ZZ Y_{\rev}$& $0$&$0$&$\ZZ_2$&?\\
 $\tCon_\blt \ZZ Y_{\rev,\trp}$& $0$&$0$&$\ZZ_2$&?
 \end{tabular}
\end{center}

The data from the table was obtained through direct computation. We provide an outline of the computations for $ Y_\rev $ below; the values for $ Y_\trp $ and $ Y_{\rev,\trp} $ were derived in an analogous, but slightly more complicated, manner.

\noindent \textbf{Calculation for $Y_\rev$:} By definition of the Yoneda embedding, $\repcub{\rev}{1}{k}$ is the set of morphisms from $[1]^k$ to $[1]^1$ in the cube category. This family of sets is a cubical set with reversals as follows: Given a morphism $\boldsymbol{f}:[1]^k\to [1]^n$ in the cube category with reversals, postcomposition gives rise to a map
\[
    f:\repcub{\rev}{1}{n}\to \repcub{\rev}{1}{k}.
\]
To make the notation more manageable, we write $\br_{i_1,\dots i_k}$ or  $\br_{\{i_1,\dots i_k\}}$ for  $\br_{i_1}\cdots \br_{i_k}$ (recall that the $\br_i$ commute) and set
$$\bg_{i_1,\dots i_k}^{\ve_1,\dots,\ve_k}\letbe \bg_{i_1}^{\ve_1}\cdots \bg_{i_k}^{\ve_k}.$$

Using the relations from \hyperref[appendix]{Appendix~A}, one can show that
every morphism in the cube category with reversals can be written, from left to right, as face maps composed with connections, composed with reversals, composed with degeneracies (see \cite[Equation (58)]{GrandisMauri}). Since there are no connections with domain $[1]^1$, only degenerate maps are obtained if there is a face map appearing in this normal form. Using this, an easy computation shows that 
\begin{align*}
    \repcub{\rev}{1}{0}&=\left\{\bd_1^1,\bd_1^0 \right\}\\
    \repcub{\rev}{1}{1}&=\left\{\id,\br_1 \right\}\cup\left\{\textnormal{degeneracies}\right\}\\
    \repcub{\rev}{1}{2}&=\left\{\bg_1^0,\bg_1^0\br_1,\bg_1^0\br_2,\bg_1^0\br_{1,2},\bg_1^1,\bg_1^1\br_1,\bg_1^1\br_2,\bg_1^1\br_{1,2}\right\}\cup\left\{\textnormal{degeneracies}\right\}\\
    \vdots
\end{align*}
The Yoneda embedding maps the morphism $\boldsymbol{r}_1$ to the morphism $(\boldsymbol{r}_1)_*$ of cubical sets defined by precomposition.
Therefore, since colimits in presheaf categories can be computed pointwise, to obtain $Y_\rev$, we identify every element $\boldsymbol{f}\in \repcub{\rev}{1}{k}$ with $\br_1 \boldsymbol{f}$.  
We denote the equivalence class of $\boldsymbol{f}$ also by $\boldsymbol{f} $.  The cubical alternating face map complex $C_\blt Y_\rev$ is then generated by these equivalence classes modulo the degeneracies. In small degrees, it can be shown that the chain complex has the following
minimal generating sets:

\begin{align*}
    C_0 Y_\rev &=\left\langle\big\{\bd_1^0 \big\}\right \rangle_\mathbb{Z} \\
    C_1 Y_\rev &=\left\langle\big\{\id \big\}\right \rangle_\mathbb{Z} \\
    C_2 Y_\rev &=\left\langle\big\{ \bg_1^0,\bg_1^0\br_1,\bg_1^0\br_2,\bg_1^0\br_{1,2}\big\}\right \rangle_\mathbb{Z} \\
    C_3 Y_\rev &=\left\langle\big\{ \bg_{1,2}^{0,0}\br_A,\bg_{1,1}^{0,1}\br_A,\bg_{1,2}^{0,1}\br_A\,\big\vert\, A\subseteq \{1,2,3\}\big\}\right \rangle_\mathbb{Z} \\
    C_4 Y_\rev &=\big\langle\big\{\bg_{1,2,3}^{0,0,0}\br_A,
    \bg_{1,2,2}^{0,0,1}\br_A,\bg_{1,2,3}^{0,0,1}\br_A,
    \bg_{1,1,1}^{0,1,0}\br_A,\bg_{1,1,2}^{0,1,0}\br_A,
    \bg_{1,1,3}^{0,1,0}\br_A,\\&\qquad~~\bg_{1,2,2}^{0,1,0}\br_A,\bg_{1,2,3}^{0,1,0}\br_A,
    \bg_{1,1,2}^{0,1,1}\br_A,\bg_{1,1,3}^{0,1,1}\br_A,\bg_{1,2,3}^{0,1,1}\br_A
    \,\big\vert\, A\subseteq \{1,2,3,4\}\big\}\big\rangle_\mathbb{Z}.
\end{align*}
Note that some additional care is required to compute the differential, since the set of representatives is not closed under precomposition with $\bd_i^\ve$.
For finding a minimal generating system of the $\{\rev\}$-symmetry sub-complex, fix Hamiltonian paths $E=(e_1,\dots ,e_8)$ and $F=(f_1,\dots ,f_{16})$ in the Hasse diagrams of the Boolean lattices of order $3$ and $4$ respectively. Then

\begin{align*}
    \rCon_0Y_\rev&=0\\
    \rCon_1Y_\rev&=\left\langle\{2\cdot \id\right\rangle\}\\
    \rCon_2Y_\rev&=\left\langle \bg_1^0+\bg_1^0\br_1,\bg_1^0\br_1 + \bg_1^0\br_{1,2} ,\bg_1^0\br_{1,2}+\bg_1^0\br_2 \right\rangle_\mathbb{Z}\\
    \rCon_3Y_\rev&=\left\langle\big\{ \bg_{1,2}^{0,0}(\br_{e_i}+\br_{e_{i+1}}),\bg_{1,1}^{0,1}(\br_{e_i}+\br_{e_{i+1}}),\bg_{1,2}^{0,1}(\br_{e_i}+\br_{e_{i+1}})\,\big\vert\, 1\leq i \leq 7\big\}\right \rangle_\mathbb{Z}  \\
    \rCon_4Y_\rev&=\big\langle\big\{\bg_{1,2,3}^{0,0,0}(\br_{f_i}+\br_{f_{i+1}}),
    \bg_{1,2,2}^{0,0,1}(\br_{f_i}+\br_{f_{i+1}}),\bg_{1,2,3}^{0,0,1}(\br_{f_i}+\br_{f_{i+1}}),
    \bg_{1,1,1}^{0,1,0}(\br_{f_i}+\br_{f_{i+1}}),\\&\qquad~~
    \bg_{1,1,2}^{0,1,0}(\br_{f_i}+\br_{f_{i+1}}),\bg_{1,1,3}^{0,1,0}(\br_{f_i}+\br_{f_{i+1}}),\bg_{1,2,2}^{0,1,0}(\br_{f_i}+\br_{f_{i+1}}),\bg_{1,2,3}^{0,1,0}(\br_{f_i}+\br_{f_{i+1}}),\\&\qquad~~
    \bg_{1,1,2}^{0,1,1}(\br_{f_i}+\br_{f_{i+1}}),\bg_{1,1,3}^{0,1,1}(\br_{f_i}+\br_{f_{i+1}}),\bg_{1,2,3}^{0,1,1}(\br_{f_i}+\br_{f_{i+1}}) \,\big\vert\, 1\leq i \leq 15\big\}\big\rangle_\mathbb{Z}.
\end{align*}

In particular, the ranks of the $\rCon_kY_\rev$ are $0,1,3,21,165,\dots$. From here, the homology groups can be calculated by brute force.

\appendix
\section{Relations among maps in symmetric simplicial and cubical structures} \label{appendix}

\subsection{Simplicial Identities}
All the relations are in $\Dsym$.
\begin{simplicial}
\label{eq:I} \bd_j\bd_i=\bd_{i+1}\bd_j  &\quad \for j\leq i
\end{simplicial}
\vskip-0.5cm
\begin{simplicial}
\label{eq:II}   \bs_j\bs_i =\bs_i\bs_{j+1} &\quad\for j\geq i
\end{simplicial}
\vskip-0.5cm
\begin{simplicial}
\label{eq:III} \bs_j\bd_i&=\begin{cases}
    \bd_{i-1}\bs_{j} &\for j<i-1\\
    \id &\for j=i,i-1\\
    \bd_i\bs_{j-1} &\for j>i
\end{cases} 
\end{simplicial}
\vskip0.1cm
\begin{simplicial}
\label{eq:IV} \bt_j\bd_i &= \begin{cases}
    \bd_i\bt_j & \for j<i-1 \\
    \bd_{i-1} & \for j=i-1\\
    \bd_{i+1} & \for j=i \\
    \bd_i\bt_{j-1} & \for j>i
\end{cases} 
\end{simplicial}

\begin{simplicial}
\label{eq:V}  \bs_j\bt_i &=\begin{cases}
 \bt_{i-1}\bs_j  & \for j< i-1\\ 
 \bt_{i-1}\bs_{i}\bt_{i-1}&\for j=i-1\\
 \bs_j &\for j=i\\
 \bt_{i}\bs_i\bt_{i+1}&\for j=i+1 \\
  \bt_i\bs_j  & \for j>i+1 
\end{cases} 
\end{simplicial}

\begin{simplicial}
\label{eq:VI}  \bt_j\bt_i  = \bt_i\bt_j \quad \for |i-j|>1 
\end{simplicial}

\begin{simplicial}
\label{eq:VII}  \bt_i^2=\id 
\end{simplicial}

\begin{simplicial}
\label{eq:VIII}  (\bt_i\bt_{i+1})^3=\id 
\end{simplicial}

\subsection{Cubical Identities}
All the relations are in $\Box^A$:

\begin{cubical}
 \label{Ceq:I}\bd_{j}^\eta\bd_{i}^\ve  =  \bd_{i+1}^\ve \bd_{j}^\eta  \quad \textnormal{for } j \leq i 
 \end{cubical}
 \begin{cubical}
 \label{Ceq:II} \bs_j \bd_{i}^\ve  &=
\begin{cases}
  \bd_{i-1}^\ve \bs_j  & \textnormal{for } j < i \\
\textnormal{id} & \textnormal{for } j = i \\
  \bd_{i}^\ve \bs_{j-1}  & \textnormal{for } j > i
\end{cases} 
\end{cubical}
\begin{cubical}
 \label{Ceq:III} \bs_i \bs_j  &=    \bs_j\bs_{i+1} \quad \textnormal{for } j \leq i 
 \end{cubical}
 \begin{cubical}
\label{Ceq:IV}\bConnection_{i}^\eta \bConnection_{j}^\ve  &=
\begin{cases}
 \bConnection_{j}^\ve \bConnection_{i+1}^\eta & \textnormal{for } j<i \\
\bConnection_{i}^\ve  \bConnection_{i+1}^\ve& \textnormal{for } j = i, \eta = \ve 
\end{cases}
\end{cubical}

\begin{cubical}
\label{Ceq:V} \bConnection_{j}^\eta \bd_{i}^\ve &=
\begin{cases}
 \bd_{i-1}^\ve \bConnection_{j}^\eta  & \textnormal{for } j < i - 1 \\
\textnormal{id} & \textnormal{for } j = i - 1, i\text{ and } \ve = \eta \\
 \bd_{j}^\ve \bs_j  & \textnormal{for } j = i - 1, i \text{ and } \ve = 1 - \eta \\
\bd_{i}^\ve \bConnection_{j-1}^\eta  & \textnormal{for } j > i\\
\end{cases} 
\end{cubical}

\begin{cubical}
\label{Ceq:VI}\bs_j \bConnection_{i}^\ve   &=
\begin{cases}
  \bConnection_{i-1}^\ve \bs_j & \textnormal{for } j < i \\
  \bs_i \bs_i & \textnormal{for } j = i \\
 \bConnection_{i}^\ve \bs_{j+1} & \textnormal{for } j > i
\end{cases}\\
\end{cubical}
\vskip-1.2cm
\begin{cubical}
\label{Ceq:VII}
  \bConnection_j^\varepsilon \bt_i &=\begin{cases}
    \bt_{i-1}\bConnection_j^\varepsilon   & \for j<i-1\\
     \bt_{i-1}\bConnection_{i}^\ve \bt_{i-1} & \for j=i-1\\
    \bConnection_j^\ve & \for j=i\\
     \bt_{i}\bConnection_i^\ve \bt_{i+1}& \for j=i+1\\
    \bt_i\bConnection_j^\varepsilon  & \for j>i+1
\end{cases}\\
\end{cubical}
\vskip-1.2cm
\begin{cubical}
\label{Ceq:VIII}
 \bt_j \bd_i^\ve  & = \begin{cases}
       \bd_i^\ve \bt_j & \for j<i-1\\
     \bd_{i-1}^\ve & \for j=i-1\\
     \bd_{i+1}^\ve & \for j=i\\
       \bd_i^\ve \bt_{j-1} & \for j>i
\end{cases}\\
\end{cubical}
\vskip-1.2cm
\begin{cubical}
\label{Ceq:IX}
  \bs_i\bt_j &= \begin{cases}
 \bt_{j}\bs_i & \for j<i-1 \\
 \bs_{i-1} & \for j=i-1 \\
 \bs_{i+1} & \for j=i \\
 \bt_{j-1}\bs_i  & \for j>i
\end{cases} \\
\end{cubical}
\vskip-1.2cm
\begin{cubical}
\label{Ceq:X}
\bt_j \bt_i &= \bt_i \bt_j  \quad \for |i-j|>1\\
\end{cubical}
\vskip-1.2cm
\begin{cubical}
\label{Ceq:XI}
 \bt_i^2&=\id\\
\end{cubical}
\vskip-1.2cm
\begin{cubical}
\label{Ceq:XII}
( \bt_i \bt_{i+1})^3&=\id
\end{cubical}
\begin{cubical}
\label{Ceq:XIII}
 \br_j \bd_i^\ve&=\begin{cases}
 \bd_i^\ve \br_j & j < i \\
\bd_i^{1-\ve} & j=i\\
\bd_i^\ve \br_{j-1} & j > i
\end{cases}
\end{cubical}
\begin{cubical}
\label{Ceq:XIV}
\bs_j \br_i  &=\begin{cases}
\br_{i-1} \bs_j & j<i\\
\bs_i & j=i\\
\br_i\bs_j  & j>i 
\end{cases}
\end{cubical}
\begin{cubical}
\label{Ceq:XV}
\br_j\bConnection_i^\ve &=\begin{cases}
\bConnection_i^\ve \br_j & j<i\\
\bConnection_i^{1-\ve}\br_j\br_{j+1} & j=i\\
\bConnection_i^\ve\br_{j+1} & j>i 
\end{cases}
\end{cubical}
\begin{cubical}
\label{Ceq:XVI}
 \bt_j \br_i &=\begin{cases}
  \br_j \bt_j & j=i-1\\
  \br_{i+1} \bt_j & j=i\\
 \br_i\bt_j  & \textnormal{else}
\end{cases}
\end{cubical}
\begin{cubical}
\label{Ceq:XVII}
\br_i \br_j &=\begin{cases}
 \id & j=i\\
 \br_j\br_i& \textnormal{else} \\
\end{cases}
\end{cubical}

\bibliographystyle{alphaurl}

\end{document}